\documentclass{amsart}
\usepackage{amsmath, amssymb}
\usepackage{pgf,pgfarrows,pgfnodes,pgfautomata,pgfheaps,pgfshade,hyperref, amssymb,enumerate,amsmath}
\usepackage[all]{xy}
\usepackage[capitalize]{cleveref}
\usepackage{mathtools}
\usepackage[shortlabels]{enumitem}
\usepackage{bbm}
\usepackage{bm}
\usepackage{csquotes}
\usepackage[margin=3.5cm]{geometry}

\newtheorem{theorem}{Theorem}[section]
\newtheorem*{theorem*}{theorem}
\newtheorem{proposition}[theorem]{Proposition}
\newtheorem{lemma}[theorem]{Lemma}
\newtheorem{corollary}[theorem]{Corollary}

\newtheorem{exem}[theorem]{Example}
\newtheorem{remark}[theorem]{Remark}

\usepackage{euscript,epsfig}
\usepackage{tikz}
\usetikzlibrary{graphs}
\usetikzlibrary[graphs]
\usetikzlibrary{arrows}
\usetikzlibrary{tikzmark}
\usetikzlibrary{positioning}
\usetikzlibrary{decorations.pathreplacing,calligraphy}
\usepackage{tikz-cd}
\usetikzlibrary{decorations.markings}
\usepackage[colorinlistoftodos]{todonotes}

\makeatletter
\newcommand{\customlabel}[2]{%
  \def\@currentlabel{#1}%
  \label{#2}%
}
\makeatother

\def\R{{\mathbb R}}
\def\Z{{\mathbb Z}}
\def\H{{\mathbb H}}

\def\N{{\mathbb N}}

\def\QQ{{\mathbb Q}}

\def \Q {{\bf Q}}

\def\Aut{{\rm Aut}} 
\def\AutDyn{{\rm Cent}} 
\def\GLN{{\rm GL}}
\def\Cent{{\rm Z}} 
\def\Centr{{\rm C}}
\def\No{{\rm N}} 

\def\Stab{{\rm Stab}}
\def\Homeo{{\rm Homeo}}
\def\SL{{\rm SL}}
\def\GL{{\rm GL}}
\def\PGL{{\rm PGL}}
\def\PSL{{\rm PSL}}

\DeclarePairedDelimiter\floor{\lfloor}{\rfloor}

\address[Jaime Gómez]{Mathematisch Instituut,
		University of Leiden,
		Einsteinweg 55,
		2333 CC Leiden, The Netherlands} \email{jaagomezortiz@gmail.com}

\address[Samuel Petite]{Laboratoire Ami\'enois
	de Math\'ematiques Fondamentales et Appliqu\'ees, CNRS-UMR 7352, Universit\'{e} de Picardie Jules Verne, 33 rue Saint Leu, 80039   Amiens cedex 1,
	France.} \email{samuel.petite@u-picardie.fr}

\thanks{J. Gómez was supported by ANID/Doctorado Nacional No. 21200054. 
He would like to thank LAMFA, Université de Picardie Jules Verne for its warm hospitality, where part of this work has been done.}

\thanks{S. Petite was supported by the ANR project IZES ANR-22-CE40-0011.  He thanks the support of the project MathAmSud240026.}



\begin{document}

\title{On the centralizer and normalizer  groups of odometers and  Toeplitz subshifts}

\author{Jaime G\'omez}
\author{Samuel Petite}
\begin{abstract} For residually finite groups, we study the automorphism and normalizer groups of odometers and their symbolic extensions: Toeplitz subshifts.
In relation to orbit equivalence theory, we show that two continuous orbit equivalent $G$-odometers have commensurable normalizers when $G=\mathbb{Z}^d$, but that this result fails for other nilpotent groups $G$.
The automorphism and normalizer groups of any Toeplitz subshift embed into those of its underlying odometer; despite this constraint, they exhibit considerable flexibility. For instance, we exhibit examples of low-complexity $G$-Toeplitz subshifts whose centralizer is isomorphic to $G$ (or, at the other extreme, restricted to its center) yet whose normalizer group is as large as possible.
\end{abstract}

\date{}

\maketitle

\section{Introduction}

Algebraic invariants associated with topological dynamical systems play an important role in their classification and study.
Examples of such invariants include centralizer groups (also called automorphism groups), normalizer groups, full groups, Ellis semigroups, and stabilized automorphism groups, among others \cite{He69,CoMe16,Au88,HaKrSc22}.
Their isomorphism classes are invariants of conjugacy and, in some cases, such as for equicontinuous systems, they even serve as complete invariants of the dynamical system.
If one considers a weaker relation than conjugacy, such as topological or measurable orbit equivalence, the corresponding isomorphism classes need no longer be invariant. In such cases, one refers to the \textit{flexibility} of these invariants, and a natural question arises: what is the range of such invariants within a single class of systems? Conversely, when these invariants remain stable (in a suitable sense) across a class of systems, one speaks of \textit{rigidity}.
This paper addresses precisely this problem: the flexibility and rigidity of automorphism and normalizer groups within the framework of continuous orbit equivalence.

Recall that two topological dynamical systems $(X,\alpha, G)$ and $(Y,\beta, H)$, for countable groups $G$ and $H$, are \textit{continuous orbit equivalent} (COE) when there is a homeomorphism $h:X\to Y$ that maps $\alpha$-orbits onto $\beta$-orbits and the associated cocycles are continuous.
The study of continuous orbit equivalence is primarily motivated by classification theory, since it provides a dynamical interpretation of isomorphism between $C^*$-algebras associated with the dynamical systems (\cite{GPS95} for $\Z$-actions, and later \cite{Renault} for general actions).
Conjugate systems are COE, but the converse fails when the phase space is disconnected. This notion is therefore particularly relevant for actions on the Cantor set. Little is known about this class, even though it is characterized by the topological full group of the action (\cite{GPS99} and \cite{Med}, for $\Z$-actions and for general actions under mild assumptions, respectively). Since there is no simple description of the full group for a general action, the understanding of this class has mainly been carried out for equicontinuous minimal dynamical systems \cite{GPS2019, CoMe16}.
As a next step, one may consider their extensions, such as Toeplitz subshifts, which form a rich family.
Moreover, the simplest topological full groups are finitely generated, and all such groups arise from subshifts \cite{Ma06}, \cite[Proposition 3.2.1]{Cornulier}.

Other groups of dynamical interest are the automorphism and normalizer groups of a system $(X,\alpha,G)$, in particular when it is a subshift. The first consists of the centralizer of $\alpha(G)$ within the group of self-homeomorphisms of $X$, i.e., the set of self-homeomorphisms commuting with every transformation of the acting group. The second consists of the normalizer of $\alpha(G)$ within the same group, i.e., the set of self-homeomorphisms that stabilize the collection of transformations of the acting group under conjugation. A normalizer map thus realizes a conjugacy of the action with itself, up to a change of generators of the group; for this reason, normalizers are also called \textit{extended symmetries}. Hence, automorphisms and normalizers are particular instances of continuous orbit equivalence. Automorphisms of equicontinuous systems are well understood, whereas no precise description is available for general Cantor systems, even for mixing subshifts (see \cite{He69,BoLiRu88} for seminal work).
For $\Z$-actions, a normalizer is also called a flip-(self)-conjugacy, and the group it forms is closely related to the automorphism group: algebraically, it is an extension of $\Z/2\Z$ by the automorphism group (see Section~\ref{Subsec: Aut-and-Norm}).
The notion of normalizer group is meaningful for more general group actions as well, but beyond $\Z$-actions, few results are known about these groups. This remains an active area of recent research; let us mention \cite{Ho10, BaCaRi25, SaSc25} for subshifts of finite type, \cite{BBHLN, MeTa19} for positive-entropy subshifts, and \cite{BaRoYa18, BLM, Ca23, CaPe23} for $\Z^d$-substitutive subshifts.

Let us summarize the main results of this paper on the rigidity and flexibility of automorphism and normalizer groups with respect to continuous orbit equivalence, a study initiated in \cite{GPS2019}.
We prove that COE odometers (for residually finite acting groups) have commensurable automorphism groups (Corollary~\ref{Coro: virtually-iso-Aut}). This follows from a rigidity result on COE odometers obtained by Cortez and Medynets \cite{CoMe16}.
In general, commensurability cannot be strengthened to isomorphism, even for $\Z^2$-actions, as Example~\ref{Example: COE-no-iso} shows.

However, we can reinforce Corollary~\ref{Coro: virtually-iso-Aut} when the acting group is free abelian and the associated odometers are torsion-free. Using a characterization from \cite{GPS2019}, we show that torsion-free COE $\Z^d$-odometers are isomorphic as abstract groups and hence have isomorphic automorphism groups (Theorem~\ref{theo:IsoCOEZdOdometer}). As for their normalizer groups, Proposition~\ref{prop:commensurableNormalizer} shows that COE $\Z^d$-odometers have commensurable normalizer groups. For non-abelian groups, the situation is entirely different, even for nilpotent groups: in Example~\ref{ex:NonQINormalizer}, we give explicit COE nilpotent odometers whose normalizer groups are not commensurable — not even quasi-isometric.

Studying automorphism and normalizer groups for the broader family of minimal symbolic systems is significantly more difficult.
The first step consists in studying Toeplitz subshifts, which are almost one-to-one extensions of odometers.
It follows that their automorphism and normalizer groups embed into those of the maximal equicontinuous factor (Proposition~\ref{Prop:MEF-implies-compatible} and Corollary~\ref{Coro:injectiveAut}); consequently, their automorphism groups are residually finite.
Beyond this constraint, the groups exhibit significant flexibility, with a degree of independence between the automorphism group and the normalizer. For instance, for the linear groups $\SL_n(\Z)$, $\PSL_n(\Z)$, $\GL_n(\Z)$, and the Heisenberg group, we construct explicit examples of $G$-Toeplitz subshifts with trivial centralizer, i.e., a centralizer consisting only of shifts by elements of the center of the group (see Section~\ref{sec:ToeplTrivialAuto}).
Note that examples of minimal subshifts over non-abelian groups with trivial centralizer were already known (e.g., the action of a free group on its boundary yields such a subshift of finite type). Here we present non-mixing examples that also apply to amenable groups.

We also construct examples with centralizer isomorphic to the acting group but with the largest possible normalizer group, even when the latter is non-abelian (Section~\ref{sec:ToeplAutoG}).
In addition, these subshift examples have very low complexity: the number of colorings of a pattern grows linearly with the size of the pattern (Proposition~\ref{prop:complexity}). These examples can be seen as an extension of the results of \cite{BBHLN, CaPe23} to the non-abelian case, and show that low complexity does not imply amenability of all the algebraic invariants.
These constructions rely on the machinery of semi-cocycle extensions: we obtain the examples by controlling the automorphism group through the set of discontinuity points of the semi-cocycle (Theorem~\ref{Theorem: diverse-Aut}).

The paper is organized as follows.
Section~\ref{Sec:Basics} recalls the notions and notation for dynamical systems, odometers, Toeplitz subshifts, and the automorphism and normalizer groups of dynamical systems, together with their main properties.
Section~\ref{Sec:AutandNorCOE} is devoted to continuous orbit equivalence and its relation to automorphism and normalizer groups within the class of minimal equicontinuous dynamical systems; we present examples of continuous orbit equivalent systems with diverse behavior in their automorphism and normalizer groups.
Section~\ref{Sec:Construction-Toeplitz} focuses on the construction underlying Theorem~\ref{Theorem: diverse-Aut}, which guarantees the existence of symbolic dynamical systems with restrictions on their automorphism group.
Finally, Section~\ref{Sec:Applications} uses the results of Section~\ref{Sec:Construction-Toeplitz} to construct dynamical systems with trivial automorphism groups for certain amenable and non-amenable acting groups.

\section{Preliminaries}\label{Sec:Basics}
In this section, we set notation and recall definitions and background on dynamical systems and their morphisms. We then focus on odometer systems and their symbolic extensions: Toeplitz subshifts.

\subsection{Dynamical systems}
Let $G$ be a countable infinite group. We denote by $\Cent(G)$ the center of $G$ (i.e., the elements commuting with every other element of $G$), and by $\Aut(G)$ its automorphism group (i.e., the collection of algebraic automorphisms of $G$).
By a \textit{(topological) dynamical system} we mean a continuous (left) action $\alpha$ of $G$ on a compact metric space $X$.
We denote this dynamical system by $(X,\alpha,G)$, where $\alpha^g:X\to X$ denotes the homeomorphism generated by this action for every $g\in G$, and by $\alpha(G)$ the group formed by these maps.
That is, $\alpha(G) = \{\alpha^g:X\to X\mid g\in G\}$.
The set $\Stab_\alpha(x)$ denotes the stabilizer of $x$ under the action $\alpha$.
We say that $(X,\alpha, G)$ is \textit{free} if $\Stab_\alpha(x) = \{1_G\}$ for every $x\in X$.
The system $(X,\alpha, G)$ is \textit{minimal} if the orbit $O_\alpha(x) = \{\alpha^g(x)\mid g\in G\}$ of every $x\in X$ is dense.
A system $(X,\alpha, G)$ is said to be \textit{equicontinuous} if the family of maps $\alpha(G)$ is equicontinuous. Equivalently, this means that there exists a metric for which each map $\alpha^g$ is an isometry.

Let $\Sigma$ be a finite set.
Consider the set $\Sigma^G=\{x:G\to \Sigma\}$, endowed with the product topology and the shift action $\sigma$ of $G$ on $\Sigma^G$ given by $\sigma^g x(h) = x(gh)$ for all $g,h\in G$.
A \textit{subshift} is a dynamical system $(X,\sigma, G)$ where $X\subseteq \Sigma^G$ is a closed, invariant set — i.e., $\sigma^g(X)=X$ for each $g\in G$ — and $\sigma$ is the shift action restricted to $X$.
A subshift is said to be \textit{minimal} if it contains no non-trivial proper subshift.
A surjective continuous map $\pi: X\to Y$ between two dynamical systems $(X,\alpha, G)$ and $(Y,\beta, G)$ is called a \textit{factor} map if $\pi\circ \alpha^g =\beta^g\circ \pi$ for every $g\in G$.
When a factor map is also injective, we call it a \textit{conjugacy}, and say that the systems $(X,\alpha,G)$ and $(Y,\beta, G)$ are conjugate.
When the system $(X, \alpha, G)$ is minimal, the factor map $\pi: X\to Y$ is \textit{almost $1$-$1$} if the set $\{y\in Y: |\pi^{-1}(\{y\})|=1\}$ is non-empty (it is in fact a dense residual subset of $Y$).
More generally, given a group isomorphism $\theta \colon G \to H$ between two groups $G$ and $H$, we call a $\theta$-\textit{conjugacy} between two actions $(X,\alpha, G)$ and $(Y,\beta, H)$ any homeomorphism $\pi \colon X\to Y$ such that $\pi \circ \alpha^g = \beta^{\theta (g)}\circ \pi$ for every $g \in G$. The notion of $\theta$-conjugacy on the same space leads to the notion of normalizer.

\subsection{Automorphism and normalizer groups}\label{Subsec: Aut-and-Norm}
Let $(X,\alpha, G)$ be a dynamical system.
The \textit{automorphism group} $\AutDyn(X,\alpha)$ of $(X,\alpha, G)$ is defined as the centralizer of $\alpha(G)$ in $\Homeo(X)$, endowed with composition.
That is,
\begin{align*}
    \AutDyn(X,\alpha) = \mbox{C}_{\Homeo(X)}(\alpha(G)).
\end{align*}
An element $\varphi\in \AutDyn(X,\alpha)$ is called an \textit{automorphism} of $(X,\alpha, G)$, and satisfies $\varphi\circ \alpha^g =\alpha^g\circ \varphi$ for every $g\in G$. For free systems, the automorphisms lying in $\alpha(G)$ correspond precisely to the (image of the) centralizer of $G$; in this sense, $\alpha(\Cent(G))$ is the smallest group one can expect the automorphism group to contain.
\begin{lemma}
    Let $(X,\alpha,G)$ be a dynamical system.
    If there exists $x_0\in X$ with trivial stabilizer $\Stab_\alpha(x_0)=\{1_G\}$,
    then $\alpha(G)\cap\AutDyn(X,\alpha)=\alpha(\Cent(G))$.
\end{lemma}
\begin{proof}
    It is straightforward to check that $\alpha(\Cent(G))\subseteq \alpha(G)\cap \AutDyn(X,\alpha)$.

    Now, let $h\in G$ be such that $\alpha^h\in \AutDyn(X,\alpha)$.
    Then $\alpha^g\alpha^h=\alpha^h\alpha^g$ for every $g\in G$, so in particular
$\alpha^g\alpha^h(x_0)=\alpha^h\alpha^g(x_0)$.
This implies $g^{-1}h^{-1}gh=1_G$, and hence $h\in \Cent(G)$.
\end{proof}

Let $(X,\alpha,G)$ be a free dynamical system, so that $\alpha(G)$ can be identified with $G$.
The \textit{normalizer group} $\No(X,\alpha)$ of $(X,\alpha,G)$ is defined as the normalizer of $\alpha(G)$ in $\mbox{Homeo}(X)$.
That is, $$\No(X,\alpha):=\mbox{N}_{\Homeo(X)}(\alpha(G)).$$
This means that any $\phi\in \No(X,\alpha)$ is a self-homeomorphism of $X$ such that for each $g\in G$, there exists $h^\phi_g\in G$ with $\phi\circ \alpha^g=\alpha^{h^\phi_g}\circ \phi$.
The map $j_\phi \colon g \in G \mapsto h^\phi_g\in G$ is an element of $\Aut(G)$:
by the freeness assumption, it is straightforward to check that it is a group homomorphism, and it is injective, since $h^\phi_g=h^\phi_{g'}$ implies $\phi\circ\alpha^g=\alpha^{h^\phi_g}\circ\phi=\alpha^{h^\phi_{g'}}\circ\phi=\phi\circ\alpha^{g'}$, hence $\alpha^g=\alpha^{g'}$. Surjectivity follows from the fact that $\No(X,\alpha)$ normalizes $\alpha(G)$.
This lets us define a group homomorphism $j:\No(X,\alpha)\to \Aut(G)$ by $j(\phi)=j_\phi$, where $\phi\circ\alpha^g = \alpha^{j_\phi (g)}\circ \phi$.
Note that $\mbox{Ker}(j)=\AutDyn(X,\alpha)$.
The group of \textit{transformations of the system} $(X, \alpha, G)$ is the image of $j$, denoted
 $$\GLN(G,\alpha) = j(\No(X, \alpha) ) \quad  (\leqslant \Aut(G)).$$
 For $\psi\in \Aut(G)$, we denote by $\No(X,\alpha,\psi)$ the subset $j^{-1}(\psi)$ of $\No(X,\alpha)$ — that is, the elements $\phi\in \No(X,\alpha)$ such that $\phi\circ \alpha^g=\alpha^{\psi(g)}\circ\phi$ for every $g\in G$. Such a map $\phi$ is called a $\psi$-\textit{automorphism}.
 We summarize the above discussion in the short exact sequence of groups
\begin{align}\label{eq: shortexactsequence}
1\to \AutDyn (X,\alpha)\to \No(X,\alpha)\overset{j}{\to} \GLN(G,\alpha)\to 1.
\end{align}

Note that restricting to inner automorphisms of the group, we can always find corresponding elements in the normalizer.
\begin{lemma}
Let $(X,\alpha,G)$ be a dynamical system.
Then $\GLN(G, \alpha)$ contains every inner automorphism.
\end{lemma}
\begin{proof}
    Let $\psi$ be the inner automorphism induced by an element $g\in G$, i.e., $\psi(h)=ghg^{-1}$ for each $h\in G$.
    Define $\phi: X\to X$ by $\phi(x)=\alpha^gx$; this map is a homeomorphism.
    Moreover, for each $h\in G$ and $x\in X$,
    \begin{align*}
        \phi(\alpha^h x) = \alpha^g(\alpha^h x) = \alpha^{ghg^{-1}}(\alpha^g x) = \alpha^{\psi(h)}\phi(x),
    \end{align*}
    which shows that $\phi\in \No(X,\alpha,\psi)$.
\end{proof}

Let $\pi: X\to Y$ be a factor map between two dynamical systems $(X,\alpha, G)$ and $(Y,\beta, G)$.
    We say that $\pi$ is \textit{compatible} with $\phi\in \mbox{Homeo}(X)$ if $\pi(\phi(x))=\pi(\phi(x'))$ whenever $\pi(x)=\pi(x')$, for every $x,x'\in X$. This allows us to define a map $\hat\pi(\phi) \colon Y \to Y$ by $\hat\pi(\phi) (\pi(x)) = \pi(\phi(x))$.
    For a subgroup $S\leqslant \mbox{Homeo}(X)$, we say that $\pi$ is \textit{compatible} with $S$ if it is compatible with every $\phi\in S$.
    In this case, $\pi$ is also compatible with $\phi^{-1}$, so $\hat\pi(\phi)$ is invertible, and the map $\hat{\pi}: \phi \in S \mapsto  \hat\pi(\phi) \in \mbox{Homeo}(Y)$ is a homomorphism.

Recall that every dynamical system $(X,\alpha, G)$ admits a \textit{maximal equicontinuous} factor $(X_{eq},\beta, G)$: an equicontinuous factor that is maximal in the sense that for any other equicontinuous factor $\pi' \colon (X, \alpha, G) \to (Y,\beta',G)$ of $(X,\alpha, G)$, there is a factor map $\zeta \colon (X_{eq},\beta,G) \to (Y,\beta', G)$ such that $\pi' = \zeta \circ \pi$.
The maximal equicontinuous factor is unique up to conjugacy. If $(X, \alpha, G)$ is an almost $1$-$1$ extension of an equicontinuous system, then that equicontinuous system is its maximal equicontinuous factor.

From these notions we derive the following result, which extends \cite[Lemma 2.3]{CaPe23}; it can likewise be generalized to other maximal factors (e.g., the maximal zero-entropy factor).
\begin{proposition}\label{Prop:MEF-implies-compatible}
Let $(X,\alpha, G)$ be a minimal dynamical system, and let $\pi:(X,\alpha, G) \to (X_{eq}, \beta, G)$ be the factor map onto its maximal equicontinuous factor.
Then:
    \begin{enumerate}
        \item $\pi$ is compatible with $\No(X,\alpha)$;
        \item $\hat{\pi}(\No(X,\alpha,\psi))\subseteq \No(X_{eq},\beta,\psi)$ for every $\psi\in \Aut(G)$;
        \item If $\pi$ is almost $1$-$1$, then the map $\hat{\pi}:\No(X,\alpha)\to \No(X_{eq},\beta)$ is injective.
    \end{enumerate}
\end{proposition}
\begin{proof}
Let $\phi \in \No(X,\alpha)$. Observe that $\pi \circ \phi \colon X \to X_{eq}$ is a factor map onto the system defined by the action $\beta'^{g} = \beta^{j(\phi)(g)}$, $g \in G$, and note that the family of maps $\beta(G)= \beta'(G)$ is equicontinuous. By maximality, there is a factor map $\zeta\colon  (X_{eq}, \beta, G) \to (X_{eq}, \beta', G)$ such that $\pi\circ \phi = \zeta \circ \pi$. It follows that $\pi$ is compatible with every $\phi \in \No(X,\alpha)$, proving (1).

Moreover, for any $\phi\in \No(X,\alpha,\psi)$ we have $$\hat{\pi}(\phi)(\beta^g\pi(x))=\hat{\pi}(\phi)(\pi(\alpha^g x))=\pi(\phi(\alpha^g x))=\alpha^{\psi(g)}(\pi(\phi(x)))=\alpha^{\psi(g)}\hat{\pi}(\phi)(\pi(x)),$$
which shows that $\hat{\pi}(\phi)\in \No(Y,\beta,\psi)$, proving (2).

For (3), assume $\pi$ is an almost $1$-$1$ factor map with $\pi(x)\in X_{eq}$ an element with a single pre-image, and suppose $\hat{\pi}(\phi)={\rm Id}_{X_{eq}}$ for some $\phi\in \No(X,\alpha,\psi)$, $\psi\in \Aut(G)$.
For every $g\in G$ we have $$\pi(\alpha^g x)=\hat{\pi}(\phi)(\pi(\alpha^g x))=\pi(\phi(\alpha^g x)).$$
Since $\pi(x)$ has a single pre-image, so does $\alpha^g \pi(x)$; consequently, $\alpha^gx=\phi(\alpha^g x)$ for every $g\in G$.
Since, by minimality, the orbit $O_\alpha(x)$ is dense in $X$, this gives $z=\phi(z)$ for every $z\in X$, i.e., $\phi={\rm Id}_X$, which proves (3).
\end{proof}
Restricting to the group $\AutDyn(X,\alpha)$ yields the following immediate corollary.

\begin{corollary}\label{Coro:injectiveAut}
Let $(X,\alpha, G)$ be a minimal dynamical system and $\pi \colon (X,\alpha, G) \to (X_{eq}, \beta, G)$ an almost $1$-$1$ factor map onto an equicontinuous system.
Then the homomorphism $\hat\pi \colon \AutDyn(X,\alpha)\to \AutDyn(X_{eq},\beta)$ is injective.
\end{corollary}

The following lemma will be key for the rest of the paper. It states that for factor maps $\pi$ compatible with $\No(X,\alpha)$, the maps $\hat\pi(\No(X,\alpha))$ preserve the cardinality of fibers.

\begin{lemma}\label{lemma: fibers}
    Let $(X,\alpha,G)$ and $(Y,\beta,G)$ be two minimal dynamical systems such that $\pi:X\to Y$ is a factor map compatible with a subgroup $S\leqslant \Homeo(X)$.
    Then, for any $\phi\in S$,
    \begin{align*}\{y\in Y: |\pi^{-1}(y)|=n\}=
        \hat{\pi}(\phi)(\{y\in Y: |\pi^{-1}(y)|=n\})\quad\text{for every }n\in \mathbb{N}.
    \end{align*}
\end{lemma}
\begin{proof}
    Let $n\in\mathbb{N}$ and $y\in\{y\in Y:|\pi^{-1}(y)|=n\}$.
    There exist $n$ distinct elements $x_1,x_2,\dots, x_n\in X$ with $\pi(x_i)=y$.
    Since $\phi\in S$, there exist $x_1',\dots,x_n'\in X$ with $\phi(x_i')=x_i$.
    Note that $\hat{\pi}(\phi)(\pi(x'_i)) = \pi(\phi(x'_i))=y$ for each $i\in\{1, \ldots, n\}$.
    Thus, for $z= \hat{\pi}(\phi)^{-1}(y) \in Y$, we have $\{x_1',\dots, x_n'\}\subseteq \pi^{-1}(\{z\})$.
    Moreover, for any $x'\in \pi^{-1}(\{z\})$ we have $\pi (\phi(x')) = \hat{\pi}(\phi)(\pi(x'))= \hat{\pi}(\phi)(z) = y$, so $\phi(x')\in\pi^{-1}(\{y\})$.
    The injectivity of $\phi$ then gives $z\in\{y\in Y:|\pi^{-1}(y)|=n\}$.
    The reverse inclusion follows the same way, replacing $\phi$ by $\phi^{-1}$.
\end{proof}

\subsection{Odometers}\label{Subsection: odometers}We recall here some background on a specific equicontinuous dynamical system called the \textit{odometer} (see \cite{CoPe08}). From now on, the group $G$ is assumed to be countable.
We say that $(G_n)_{n\geq 0}$ is a \textit{scale} if it is a decreasing sequence of normal, finite-index subgroups of $G$ with $\bigcap_n G_n = \{1_G\}$; in particular, this forces $G$ to be residually finite.
The \textit{(exact) odometer} associated with $(G_n)_{n\geq 0}$ is the system $( G_{(G_n)},\alpha,G)$ (or simply $G_{(G_n)}$), defined as
\begin{align*}
     G_{(G_n)}:= \{(G_n\xi_n)_{n\geq 0}\in \prod_{n=0}^\infty G/ G_n  \mid \xi_n G_n= \xi_{n+1}G_n, n\geq 0\},
\end{align*}
where $\alpha$ is the component-wise multiplication of $G$ on $ G_{(G_n)}$.
This system is an equicontinuous, minimal, free dynamical system, equipped with a group structure induced by that of $\prod_{n\geq 0}G/G_n$, which contains a dense copy of $G$ via the injective group homomorphism $\tilde{}: G\to  G_{(G_n)}$, $\tilde{g}=(gG_n)_{n\geq 0}$. We denote by $\tilde{G}$ the image of this map in $G_{(G_n)}$.

For $n\in\mathbb{N}$, we denote by $C_n$ the \textit{cylinder of level $n$} of $ G_{(G_n)}$ centered at $\tilde{1_G}= 1_{G_{(G_n)}}$, i.e.,
\begin{align}\label{eq: cylinders}
    C_n=\{(\xi_nG_n)_{n\geq 0}\in G_{(G_n)}: \xi_nG_n= G_n\}.
\end{align}
It is straightforward to check that the topology of $ G_{(G_n)}$ is generated by the collection $\{\alpha^gC_n: g\in G/G_n, n\geq 0\}$.

We denote by $\rho$ the metric on $G_{(G_n)}$ given by
\begin{align}\label{eq: metric-odometer}
    \rho((\xi_nG_n)_{n\geq 0},(\zeta_nG_n)_{n\geq 0})=2^{-k}, \mbox{ where }k=\sup\{i\geq 0: \xi_iG_i=\zeta_iG_i\}.
\end{align}
Note that $\rho$ is an ultrametric; in particular, every triple of points forms an isosceles triangle.

\subsection{Semicocycle systems and Toeplitz subshifts}\label{Subsec: Semicocycles and Toeplitz}
In this subsection, we present some facts about semicocycle extensions of odometers that will be useful throughout the paper.
A detailed presentation of this topic, in a more general framework, can be found in \cite{FuKw19}.

Let $(G_n)_{n\geq 0}$ be a scale and $\bm{\xi}\in G_{(G_n)}$.
The pair $( G_{(G_n)}, \bm{\xi})$ is called a \textit{pointed dynamical system}.
Let $K$ be a compact Hausdorff space and $(K^G, \sigma, G)$ the shift space with alphabet $K$, where $\sigma:G\to \Homeo(K^G)$ is defined by $\sigma^g(y)(h)=y(hg)$ for every $g,h\in G$; this defines a (left) continuous action of $G$ on $K^G$.
A \textit{semicocycle} over a pointed system $( G_{(G_n)}, \bm{\xi})$ is a continuous map $f:O_\alpha(\bm{\xi})\to K$.
We denote by $F$ the closure of the graph of $f$, i.e.,
$$F=\overline{\{(\bm{\zeta},k)\in  G_{(G_n)}\times K: \bm{\zeta}\in O_\alpha(\bm{\xi}), k=f(\bm{\zeta}) \}}.$$
For $\bm{\zeta}\in  G_{(G_n)}$, the \textit{$\bm{\zeta}$-section} is the set $F(\bm{\zeta})=\{k\in K: (\bm{\zeta},k)\in F\}$.
The semicocycle $f$ is said to be \textit{invariant under no rotation} if $F(\tilde g\bm{\zeta}_1)= F(\tilde g \bm{\zeta}_2)$ for every $g\in G$ implies $\bm{\zeta}_1=\bm{\zeta}_2$.
If $f$ is a semicocycle over $( G_{(G_n)}, \bm{\xi})$, we denote by $D_f$ the set of discontinuity points of $f$, i.e.,
\begin{align*}
    D_f=\{\bm{\zeta}\in G_{(G_n)}: |F(\bm{\zeta})|>1\},
\end{align*}
and by $X_{f,\bm{\xi}}$ the \textit{dynamical system associated with $f$ and $\bm{\xi}$} — the subshift of $K^G$ given by
\begin{align}\label{eq:semicocyclesystem}
    X_{f,\bm{\xi}}=\overline{O_\sigma(x_{f, \bm{\xi}})}\subseteq K^G, \mbox{ where }x_{f,\bm{\xi}}(h)=f(\tilde h\bm{\xi}), h\in G.
\end{align}
When the pointed system is clear from context, we drop the subscript $\bm{\xi}$ from $X_{f,\bm{\xi}}$ and $x_{f,\bm{\xi}}$.
\begin{lemma}[{\cite[Theorem 6]{CoPe08}}{\cite[Theorem 2.5]{FuKw19}}]\label{lemma:invariantnorotation} Let $( G_{(G_n)}, \bm{\xi})$ be a pointed system and $f$ a semicocycle over $( G_{(G_n)}, \bm{\xi})$.
If $f$ is invariant under no rotation, then $X_{f,\bm{\xi}}$ is a minimal almost automorphic extension of $ G_{(G_n)}$.
\end{lemma}

Let $(G_n)_{n\geq 0}$ be a scale and $(G_{(G_n)},\bm{\xi})$ a pointed dynamical system.
A system $(X_{f,\bm{\xi}},\sigma, G)$ is called a \textit{semicocycle extension} of $G_{(G_n)}$.
Recall that the factor map $\pi:X_{f,\bm{\xi}}\to G_{(G_n)}$ is given as follows.
Let $x\in X_f$. There exists a sequence $(h_n)_{n\geq 0}\subseteq G$ such that $(f(\alpha^g\alpha^{h_n}\bm{\xi}))_{n\geq 0}$ converges to $x(g)$ for each $g\in G$.
Let $\bm{\zeta}_x$ be an accumulation point of $(\alpha^{h_n}\bm{\xi})_{n\geq 0}$, and set $\pi(x)=\bm{\zeta}_x$.
Then $x(g)\in F(\alpha^g\bm{\zeta}_x)=F(\alpha^g\pi(x))$.
See \cite{FuKw19} for a proof of these facts.

\begin{proposition}
Let $(X_{f,\bm{\xi}},\sigma,G)$ be a semicocycle extension of $ G_{(G_n)}$.
It holds true that
$\bigcup_{g\in G}\alpha^gD_f=\{\bm{\zeta}\in  G_{(G_n)}:|\pi^{-1}(\bm{\zeta})|>1\}$.
\end{proposition}
\begin{proof}   
If $\bm{\zeta}\in \{\bm{\zeta}\in  G_{(G_n)}:|\pi^{-1}(\bm{\zeta})|>1\}$, there exist distinct elements $x,y\in\pi^{-1}(\{\bm{\zeta}\})$.
Consequently, for some $g\in G$ we have that $x(g)\neq y(g)$.
As $x(g),y(g)\in F(\alpha^g\pi(\bm{\zeta}))$, we conclude that $\alpha^g\pi(\bm{\zeta})\in D_f$ and hence, $\pi(\bm{\zeta})\in \alpha^{g^{-1}}D_f$.

On the other side, suppose that $\bm{\zeta}\in\bigcup_{g\in G}\alpha^g D_f$.
Therefore, for some $h\in G$, we have that $\bm{\zeta}\in \alpha^hD_f$. 
This implies that $\alpha^{h^{-1}}\bm{\zeta}\in D_f$ and hence, $|F(\alpha^{h^{-1}}\bm{\zeta})|>1$.
Thus, $(\alpha^{h^{-1}}\bm{\zeta},k_0),(\alpha^{h^{-1}}\bm{\zeta},k_1)\in F$, for some distinct elements $k_0,k_1\in K$.
So, there exist $(g_n)_{n\geq 0}$ and $(h_n)_{n\geq 0}$, sequences of elements in $G$, such that $(\alpha^{g_n}\bm{\xi},f(\alpha^{g_n}\bm{\xi}))\to (\alpha^{h^{-1}}\bm{\zeta},k_0)$ and $(\alpha^{h_n}\bm{\xi},f(\alpha^{h_n}\bm{\xi}))\to (\alpha^{h^{-1}}\bm{\zeta},k_1)$
 as $n\to\infty$.
Possibly, by taking subsequences, we can assume that $(\sigma^{g_n}x_{f,\bm{\xi}})_{n\geq 0}$ and $(\sigma^{h_n}x_{f,\bm{\xi}})_{n\geq 0}$ converge to some elements $x$ and $y$ in $X_{f,\bm{\xi}}$, respectively. 
By continuity of $\pi$, we conclude that $\pi(x)=\pi(y)=\alpha^{h^{-1}}\bm{\zeta}$.
Moreover, $x(1_G)=\lim_{n\to\infty}f(\alpha^{g_n}\bm{\xi})=k_0$ and $y(1_G)=\lim_{n\to\infty}f(\alpha^{h_n}\bm{\xi})=k_1$. 
Consequently, $|\pi^{-1}(\bm{\zeta})|>1$.
\end{proof}

\begin{lemma}\label{lemma: equal-cocycles}
Let $f\colon O_\alpha(\bm{\xi})\to K$ be a semicocycle  over $( G_{(G_n)},\bm{\xi}) $ that is invariant under no rotation,  and $\bm{\xi}'\in  G_{(G_n)}$ be such that $O_\alpha(\bm{\xi}')\cap D_f=\emptyset$. 
Then $f$ is a semicocycle over $( G_{(G_n)},\bm{\xi}')$ and $X_{f,\bm{\xi}}=X_{f,\bm{\xi}'}$.  
\end{lemma}
\begin{proof}
Let $x_{f,\bm{\xi}}\in X_{f,\bm{\xi}}$ and $x_{f,\bm{\xi}'}\in X_{f,\bm{\xi}'}$ be the elements associated with $f$ in $X_{f,\bm{\xi}}$ and $X_{f,\bm{\xi}'}$, respectively.
    It means, for each $s\in G$, $x_{f,\bm{\xi}}(s)=f(\tilde s\bm{\xi})$ and $x_{f,\bm{\xi}'}(s)=f(\tilde s\bm{\xi}')$.
    
    That $f$ is a semicocycle over $( G_{(G_n)},\bm{\xi}')$ is clear, since $O(\bm{\xi}')\cap D_f=\emptyset$ and to be invariant under rotation does not depend on the pointed system.
    On the other hand, as $ G_{(G_n)}$ is minimal, there exists a sequence $(g_n)_{n\geq 0}\subseteq G$ such that $\alpha^{g_n} \bm{\xi}\to \bm{\xi}'$. 
    Consequently, for every $s\in G$, we have that  $\alpha^s\alpha^{g_n} \bm{\xi}\to \alpha^s\bm{\xi}'$.
    As $f$ is continuous in $\alpha^s\bm{\xi}'$, we deduce that $f(\alpha^s\alpha^{g_n} \bm{\xi})\to f(\alpha^s\bm{\xi}')$, i.e., 
    $\sigma^{g_n}x_{f,\bm{\xi}}(s)\to x_{f,\bm{\xi}'}(s)$ and hence, $X_{f,\bm{\xi}}=X_{f,\bm{\xi}'}$ by minimality of $X_{f,\bm{\xi}}$.
\end{proof}

Let $\Sigma$ be a finite set.
We say that a subshift $X\subseteq \Sigma^G$ is a \textit{Toeplitz subshift} if it is minimal and there exist a scale $(G_n)_{n\geq 0}$, and an almost $1$-$1$ factor map $\pi: X\to  G_{(G_n)}$. 
From the previous discussion we obtain that, if $(G_n)_{n\geq 0}$ is a scale and $X_{f,\bm{\xi}}\subseteq \Sigma^G$ is a semicocyle extension of $G_{(G_n)}$, for some finite set $\Sigma$, then $X_{f,\bm{\xi}}$ is a Toepliz subshift.
For a more detailed presentation regarding Toeplitz subshifts see, for instance, \cite{CoPe08}.

\subsection{Automorphism and Normalizer groups for odometers}\label{sec:Normali} 

In this subsection, we recall and establish some facts about automorphisms and normalizers of odometers.

We first recall a characterization of the automorphisms of odometers.
\begin{proposition}\label{prop:right-aut}
Let $(G_n)_{n\geq 0}$ be a scale and $(G_{(G_n)},\alpha, G)$ be its odometer associated.
    It holds true that $$\AutDyn(G_{(G_n)},\alpha)=\{R_{\bm{\zeta}}:G_{(G_n)}\to G_{(G_n)}\mid\bm{\zeta}\in G_{(G_n)}\},$$ where $R_{\bm{\zeta}}$ is the map of right multiplication by $\bm{\zeta}$, i.e., $R_{\bm{\zeta}}(\bm{\eta})=\bm{\eta}\bm{\zeta}$, for $\bm{\eta}\in G_{(G_n)}$.
In particular, $\AutDyn(G_{(G_n)},\alpha)$ is isomorphic to $G_{(G_n)}$.
\end{proposition}
\begin{proof}
Observe that $\{R_{\bm{\zeta}}:G_{(G_n)}\to G_{(G_n)}\mid\bm{\zeta}\in G_{(G_n)}\}\subseteq \AutDyn( G_{(G_n)},\alpha)$.
Conversely, let $\phi\in\AutDyn( G_{(G_n)},\alpha)$.
    Consider $\bm{\xi}=\phi(\tilde{1_G})$.
    We have that for every $g\in G$, $\phi(\tilde{g})=\phi(\alpha^g(\tilde{1_G}))=\alpha^{g}\bm{\xi}=\tilde{g}\bm{\xi}=R_{\bm{\xi}}(\tilde{g})$.
    By minimality of $ G_{(G_n)}$, we conclude.
\end{proof}

The next proposition gives a necessary and sufficient condition for the transformations of odometers. It  will be used throughout the paper. We include its proof, since it offers useful insight for later parts of the paper.

\begin{proposition}[{\cite[Lemma 2.6]{CaPe23}}]\label{lemma: if-only-Norm}
Let $( G_{(G_n)},\alpha, G)$ be an odometer. Then it holds true that
\[\GLN(G,\alpha) = \{ \psi \in \Aut(G) \mid \forall n>0, \exists m(n)>n  \textrm{ s.t. } \psi^{-1}(G_m), \psi(G_m)\leqslant G_n\}.
\]
\end{proposition}
\begin{proof}
 Let $\psi \in \GLN(G, \alpha)$ and let $\phi: G_{(G_n)}\to  G_{(G_n)}$ be a homeomorphism that satisfies $\phi(\alpha^h \bm{\xi})=\alpha^{\psi(h)}\phi(\bm{\xi})$ for each $h\in G$ and $\bm{\xi}\in  G_{(G_n)}$.
 Let $n\in\mathbb{N}$ and let $C_n$ denote the cylinder of level $n$ centered in $\tilde{1_G}$ as in \eqref{eq: cylinders}. 
Since $\phi^{-1}(C_n)$ is a clopen set of $ G_{(G_n)}$, there exists $m\in \mathbb{N}$ such that $\alpha^dC_m\subseteq \phi^{-1}(C_n)$, for some $d\in G$.
 Let $\gamma\in G_m$.
 For any $\bm{\xi}\in \alpha^dC_m$ it holds that $\phi(\alpha^\gamma \bm{\xi})=\alpha^{\psi(\gamma)}\phi(\bm{\xi})\in C_n$.
 Using that $\phi(\bm{\xi})\in C_n$, we obtain that $\psi(\gamma)\in G_n$, which implies that $\psi(G_m)\leqslant G_n$. The same statement for $\psi^{-1}$ is true by considering $\phi^{-1}$ and $G_m$.
 Since the groups $(G_n)_n$ are  nested, we get a common index $m$ for both $\psi$ and $\psi^{-1}$.
 
 Conversely, assume that for every $n\in \mathbb{N}$, there exists $m(n)\in\mathbb{N}$ such that $\psi(G_{m(n)})$, $\psi^{-1}(G_{m(n)})\leqslant G_n$.
 This gives the existence of the map $\pi_n: G_{(G_n)}\to G/G_n$ given by $\pi_n(\bm{\xi})=\psi(\bm{\xi}_{m(n)})G_n$, $\bm{\xi} = (\bm{\xi}_nG_n)_{n\geq 0}\in G_{(G_n)}$.
 The universal property of $ G_{(G_n)}$ as inverse limit guarantees the existence of a continuous map $\phi: G_{(G_n)}\to  G_{(G_n)}$ such that $\phi \circ \alpha^g= \alpha^{\psi(g)} \circ \phi$. Similarly, we get a continuous map $\phi' \colon  G_{(G_n)}\to  G_{(G_n)}$ such that $ \phi' \circ \alpha^g= \alpha^{\psi^{-1}(g)} \circ \phi'$. Hence the map $\phi \circ \phi'$ is continuous and commutes with the action $\alpha$. Since equicontinous systems are coalescent \cite[Chapter 5]{Au88}, $\phi \circ \phi'$ is invertible and the map $\phi$ is a homeomorphism. 
\end{proof}
As a consequence, the odometers associated with a sequence and with any of its subsequences are conjugate.

\begin{exem}[The universal odometer]
 Let $G$ be a countable, residually finite group such that the family of finite-index normal subgroups of $G$ is countable.
Let $(G_n)_{n\geq 0}$ be an enumeration of this family (not necessarily nested).
Define $\Gamma_0:=G_0$, and for every $n\geq 1$, set $G_n:=\Gamma_{n-1}\cap G_n$.
It is straightforward to check that every odometer over $G$ is a factor of $(G_{(G_n)},\alpha, G)$.
This odometer $ G_{(G_n)}$ is of particular interest, since $\GLN( G,\alpha)=\Aut(G)$.
\end{exem}

In fact, the normalizer group of an odometer can be described as a semidirect product of its transformation group with its automorphism group.
\begin{proposition}\label{coro:splitsequence}
The short exact sequence given in \eqref{eq: shortexactsequence} for $( G_{(G_n)},\alpha, G)$,
\begin{align}\label{eq:splitsequence}
    1\to \AutDyn( G_{(G_n)},\alpha)\to \No( G_{(G_n)},\alpha)\to \GLN( G,\alpha)\to 1,
\end{align}
splits.
\end{proposition}

\begin{proof}
For $\psi\in \GLN( G,\alpha)$,  let $\phi\in \No( G_{(G_n)},\alpha,\psi)$.
  Consider the map $B_\psi:=\phi\circ R_{\phi(\tilde{1_G})}\in \No( G_{(G_n)},\alpha,\psi)$, where $R_{\phi(\tilde{1_G})}$ is the right multiplication automorphism defined in Proposition \ref{prop:right-aut}.
 We claim that $$B=\{B_\psi\in \No( G_{(G_n)}, \alpha): \psi\in \GLN( G,\alpha)\}$$ is a subgroup of $\No( G_{(G_n)},\alpha)$.
  Actually, first of all, note that $B_{\psi}$ is the unique map in $\No( G_{(G_n)},\alpha,\psi)$ satisfying $B_{\psi}(\tilde{g})=\widetilde{\psi(g)}$ for every $g\in G$.
  Therefore, for any $\psi_1,\psi_2\in \GLN( G,\alpha)$, we have $B_{\psi_1}\circ B_{\psi_2}(\tilde{g})=\widetilde{\psi_1(\psi_2(g))}$ for each $g\in G$.
  Consequently, $B_{\psi_1}\circ B_{\psi_2}=B_{\psi_1\circ \psi_2}$.
  Furthermore, we note that $(B_\psi)^{-1}=B_{\psi^{-1}}\in B$.

  From the previous computation, the map $\lambda:\GLN( G,\alpha)\to B$ defined by $\lambda(\psi)=B_{\psi}$ is a group homomorphism in such a way that splits the short exact sequence \eqref{eq:splitsequence}.
\end{proof}
\begin{remark}\label{rem:desc-B_psi}
    {\rm The explicit form of the maps $B_\psi$ will be important later in Section \ref{Sec:Applications}.
    For that reason, we give a more precise description of these maps.
    Let $\psi\in \GLN( G,\alpha)$, $B_\psi$ be as in the previous proof, and  $\bm{\xi}=({\xi}_nG_n)_{n\geq 0}\in  G_{(G_n)}$.
    Note that $\tilde{{\xi}_n}\to \bm{\xi}$ as $n\to\infty$.
    Thus, $B_\psi(\tilde{{\xi}_n})\to B_\psi(\bm{\xi})$ as $n\to\infty$.
    Let us denote $B_\psi(\bm{\xi})=(\eta_nG_n)_{n\geq 0}$.
    The last convergence says that for every $n\in \mathbb{N}$, there exists $m_n\in \mathbb{N}$ such that $\psi({\xi}_{m_n})G_n=\eta_nG_n$.
        Then, $B_\psi(\bm{\xi})=(\psi({\xi}_{m_n})G_n)_{n\geq 0}$ for some subsequence $({\xi}_{m_n})_{n\geq 0}$ of $(\xi_n)_{n\geq 0}$.
    
        When $\psi$ satisfies $\psi(G_n)=G_n$ for every $n\geq 0$, we can take the subsequence $(\xi_{m_n})_{n\geq 0}$ as the exact same sequence $(\xi_{n})_{n\geq 0}$. 
    }
\end{remark}

For abelian odometers, one can give a more precise description of the normalizer. We follow the presentation in \cite{GPS2019}, where normalizers are called isomorphisms.
Let $d \ge 1$ be an integer and let $(H_i)_{i\ge 1}$ be a scale of $\Z^d$.
Set, for each index $i$, the dual lattice
\[ H_i^* = \{ {n} \in \R^d: \langle {n}, {k} \rangle \in \Z, \textrm{ for all } {k} \in H_i\}
\]
and
\begin{align}\label{eq:SuperGroup}
    H^*_\infty = \bigcup_{i=1}^\infty H_i^*.
\end{align} 

It is straightforward to check that $H^*_\infty$ is an abelian group such that $\Z^d \leqslant H^*_\infty \leqslant \mathbb Q^d$. Actually, this group is an important invariant for normalizers, as shown by the  following proposition.

\begin{proposition}[{\cite[Proposition 2.8]{GPS2019}}]\label{prop:charNormalizerAbelianOdo} 
Let  $(H_n)_{n\geq 0}$ be a scale and $(\Z^d_{(H_n)},\alpha, \Z^d)$ its odometer associated.
Then we have \[ \GLN({\Z^d}, \alpha) = \{\psi \in {\rm GL}_d(\Z): \psi(H^*_\infty) = H^*_\infty\}. \]
\end{proposition}
When the scale $(H_n)$ is defined by a non-singular integer matrix $A$ via $H_n= A^n \Z^d$, we refer to Sabitova \cite{Sabitova24} for an arithmetic description of $\GLN({\Z^d}, \alpha)$, and to \cite{Sabitova25} for an algorithmic one.

\section{Automorphism and Normalizer groups for COE odometers}\label{Sec:AutandNorCOE}

In this section, we study the flexibility/rigidity of the normalizer group within a  continuous orbit equivalence class of odometers. Actually, there are rigidity properties for action of abelian group, but as soon as the commutativity assumption is dropped (e.g. for nilpotent group) the commensurability of normalizer groups for COE systems is no longer true.

Recall that two free actions $(X, \alpha, G)$ and $(Y, \beta, H)$ of two countable groups $G$ and $H$ are \textit{continuously orbit equivalent} (COE) if there is a homeomorphism $h \colon X \to Y$ and continuous functions $c_\alpha \colon X \times G \to H$ and $c_\beta \colon X \times H \to G$ such that 
\[ h(\alpha^{g} x) =  \beta^{c_\alpha (x,g) }(h(x)),\mbox{ and }\beta^{\ell} (h(x)) =  h \left(\alpha^{c_\beta (x,\ell)}(x)\right)
\]
for all $x \in X$, $g\in G$ and $\ell \in H$. In particular the function $h$ maps the $\alpha$-orbits onto $\beta$-orbits: $h(O_\alpha (x)) = O_\beta (h(x))$.

We observe that a normalizer $\phi$ of an action $(X, \alpha, G)$ defines a self continuous orbit equivalence of $(X, \alpha, G)$ with (constant) function $c_\alpha (x, g) = \psi_\phi (g)$ for all $g \in G $, $x\in X$.

The COE relation has some rigidity properties according to the acting group. M. Boyle showed  that two minimal $\Z$-actions on the Cantor set are flip conjugate, meaning they are conjugate or one map is conjugate to the inverse map of the other one.  For other groups such kind of rigidity is under investigation. Few things are known, but for free actions COE can  only be realized between actions of  quasi-iometric groups \cite{Li}. Note there is still no characterization of the COE class.

Also the type of actions restricts the COE class. M.I. Cortez and K. Medynets showed that COE odometers have to be \textit{commensurable}, i.e., they have finite index subgroups that are isomorphic. 
\begin{theorem}[{\cite[Theorem 1.2]{CoMe16}}]\label{Theorem: coe-iff-structure}
    Let $(X,\alpha,G)$ and $(Y,\beta,H)$ be two odometers.
    Then $(X,\alpha,G)$ and $(Y,\beta,H)$ are continuous orbit equivalent if and only if, there exist scales $(G_n)_{n\geq 0}$ and $(H_n)_{n\geq 0}$ determining the structure of $(X,\alpha,G)$ and $(Y,\beta,H)$ as odometers, respectively, and a group isomorphism $\theta : H_0 \to G_0$ such that $[G : G_0] = [H : H_0]$ and $\theta (H_n) = G_n$ for every $n \geq  0$.
\end{theorem}
Note that this theorem can be reinforced when one of the acting group is nilpotent. S. Hurder and O. Lukina show COE nilpotent odometers (not necessarily with a free action)  only occurs between commensurable nilpotent groups \cite{HurderLukina}.

As a consequence of Theorem \ref{Theorem: coe-iff-structure}, any two COE odometers have commensurable automorphism groups.
\begin{corollary}\label{Coro: virtually-iso-Aut}
    If $(X,\alpha, G)$ and $(Y,\beta, H)$ are two odometers that are continuous orbit equivalent,
    then $\AutDyn(X,\alpha)$ and $\AutDyn(Y,\beta)$ are commensurable.
\end{corollary}
\begin{proof}
By Theorem \ref{Theorem: coe-iff-structure} there exist scales $(G_i)_{i\geq 0}$ and $(H_i)_{i\geq 0}$ associated with $X$ and $Y$, respectively, such that there is an isomorphism $\theta: H_0\to G_0$ with $\theta(H_i)=G_i$ and $[G:G_0]=[H:H_0]$.
   First, note that
    \begin{align}\label{eq: cylinders firstlevel}
        C_0^X =\{(g_i G_i)_{i\geq 0}\in X :g_0 \in G_0\}\mbox{ and }C_0^Y =\{(h_i H_i)_{i\geq 0}\in Y :h_0 \in H_0\}
    \end{align}
     are normal subgroups of $X$ and $Y$, respectively.
      Now, using the properties of the map $\theta : H_0\to G_0$ we have that the map $\widehat{\theta}: C_0^Y\to C_0^X$, $\widehat{\theta}((h_iH_i)_{i\geq 0}) = (\theta(h_i)G_i)_{i\geq 0}$, is a group isomorphism.
    Since $X/C_0^X\cong G/G_0$, and $Y/C_0^Y\cong H/H_0$, we deduce that $[X:C_0^X] = [Y:C_0^Y]<\infty$. Hence,  the respective automorphism groups (\cref{prop:right-aut}) have isomorphic finite index subgroups. 
\end{proof}
When the acting group is $\Z$, Theorem \ref{Theorem: coe-iff-structure} and Proposition \ref{prop:right-aut} tell us the following.
\begin{corollary}
    If $(X,\alpha, \Z)$ and $(Y,\beta,\Z)$ are two COE odometers, then their automorphism groups are isomorphic, i.e., $$\AutDyn(X,\alpha)\cong \AutDyn(Y,\beta).$$
\end{corollary}

When the acting group is a free finitely generated abelian group $G = \Z^d$, (with $d\geq 2$), Example \ref{Example: COE-no-iso} shows that the previous corollary does not hold  in general.
To do that, we present a description of $\Z^d$ odometers as abstract groups. 
We start with  $\Z$ odometers.
Let $(G_n)_{n\geq 0}$ be a nested sequence of finite index subgroups of $\Z$. That is  $G_n=m(n)\Z$ for some $m(n)\in \Z^+$
and  $m(n)$ divides $m(n+1)$ for every $n\geq 0$.
Denote by $\mathcal{P}_{(G_n)}$ the set of prime numbers that divide $m(n)$ for some $n\geq 0$.
That is,
\begin{align*}
    \mathcal{P}_{(G_n)}=\{p\mbox{ prime}:p\;|\;m(n) \mbox{ for some }n\geq 0\}.
\end{align*}
Also, for every prime number $p\in\Z^+$ and $n\geq 0$ define $r_n(p)$ as
\begin{align*}
    r_n(p):=\max\{k\in \Z^+: p^k\;|\;m(n)\}, \mbox{ and }r(p)=\max_{n\geq 0}r_n(p).
\end{align*}
Observe that $p\in \mathcal{P}_{(G_n)}$ if and only if $r(p)>0$.
Using the previous notation, we have the following property.
\begin{proposition}\label{Lemma: decomposition-odometers}
    Let $(G_n)_{n\geq 0}$ be a decreasing sequence of finite index of $\Z$ such that $G_n=m(n)\Z$ with $m(n)\in \Z^+$.
    Then,  as abstract groups,
    \begin{align*}
        \Z_{(G_n)}\cong \prod_{\substack{p\in \mathcal{P}_{(G_n)}\\r(p)=\infty}}\Z_p\times \prod_{\substack{p\in\mathcal{P}_{(G_n)}\\r(p)<\infty}}\Z/p^{r(p)}\Z,
    \end{align*}
    where $\Z_p$ denotes the $p$-adic integers.

    Moreover, $\Z_{(G_n)}$ is a torsion-free group if and only if $r(p)=\infty$ for every $p\in \mathcal{P}_{(G_n)}$.   
\end{proposition}
    \begin{proof}
For every $p\in \mathcal{P}_{(G_n)}$, consider $G_n^p:=p^{r_n(p)}\Z$.
Using the fundamental theorem for finitely generated torsion abelian groups, we have that $\Z/G_n\cong \prod_{p\in \mathcal{P}}\Z/G_n^p$.
Note that for every $n\geq 0$, there exist only finitely many $p\in \mathcal{P}_{(G_n)}$ such that $G_{n}^p\neq \Z$.
Observe that for $p\in \mathcal{P}_{(G_n)}$ we have that 
\begin{align*}
    \varprojlim_n(\Z/G_n^p, \varphi_n^p)\cong\begin{cases}
    \Z/p^{r(p)}\Z &\mbox{ if }r(p)<\infty\\
    \Z_p&\mbox{ if }r(p)=\infty.
    \end{cases}
\end{align*}
    
As inverse limits and products are particular cases of limits in the category of abelian groups, they commute (see, for instance, \cite[Proposition 6.2.8]{Le14}) and we obtain the claimed isomorphism.

The second part of the Lemma follows since the components $\Z_p$ are torsion-free and we have only these components exactly when $r(p)=\infty$ for every $p\in \mathcal{P}_{(G_n)}$.  
\end{proof}

For higher rank $\Z^d$-odometer with  $d\geq 2$, we will inductively use the following property. 
Recall that for a finite index subgroup $\Gamma\leq \Z^d$, using the invariant factor decomposition for finitely generated abelian torsion groups, there is an isomorphism 
\[\Z^d/\Gamma\cong \Z/m_1\Z\times\cdots\times \Z/m_d\Z\]
where $m_1, \ldots m_d$ are $d$ positive integers,  $m_i$ divides $m_{i+1}$ for each $1\leq i <d$ and $m_1\cdots m_d=[\Z^d:\Gamma]$.

Now, using that $(G_n)_{n\geq 0}$ is a scale of $\Z^d$,  the former decomposition satisfies for each $n\ge 0$, $\Z^d/G_n\cong \Z/m_1(n)\Z\times \cdots\times \Z/m_d(n)\Z$ where 
\begin{enumerate}[(\roman*)]
    \item $m_i(n)$ divides $m_{i+1}(n)$ for every $1\leq i<d$;
    \item $m_i(n)$ divides $m_i(n+1)$ for each $n\geq 0$ and $1\leq i\leq d$.
\end{enumerate}
Only the item (ii) is not obvious from the invariant factor decomposition theorem, but it is a standard fact. See for instance exercice 4 in Section 3.9  in \cite{Jacobson} where the proof follows the same strategy as the one of the invariant factor decomposition theorem. The main idea consists in considering the abelian quotient group as $\Z/p\Z$-vector space and then to compute their dimensions.

Set then for every $1\leq i\leq d$, $(G_{n,i})_{n\geq 0}$ to be  the  sequence of nested finite index subgroups of $\Z$, where $G_{n,i}=m_i(n)\Z$. Beware it is not necessarily a scale since it could stabilize.

For every $1\leq i\leq d$ denote $\mathcal{P}^i_{(G_n)}=\mathcal{P}_{(G_{n,i})}$ and $r^i_{(G_n)}(p)=r_{(G_{n,i})}(p)$ for every prime number $p\in \Z^+$.
We have a multidimensional version of Proposition  \ref{Lemma: decomposition-odometers}, which is proved in a similar way.
\begin{proposition}\label{lemma:multidimensional-decomposition-odometers}
    Let $d\geq 1$.
    Let $(G_n)_{n\geq 0}$ be a scale of $\Z^d$.
    Then, with the previous notation there is the following abstract groups isomorphism
    \begin{align*}
        \Z^d_{(G_n)}\cong \prod_{i=1}^d \Z_{(G_{n,i})}.
    \end{align*}
    Moreover, $\Z^d_{(G_n)}$ is a torsion-free group if and only if for each $1\leq i\leq d$ and any prime number $p\in \Z^+$, $r^i_{(G_n)}(p)=\infty$ for every $p\in \mathcal{P}_{(G_n)}^i$.
\end{proposition}
We present an example with two COE odometers that are not isomorphic as abstract groups.
\begin{exem}\label{Example: COE-no-iso}
Let the $(G_n)_{n\geq 0}$ and $(H_n)_{n\geq 0}$ be the scales in $\Z^2$ given by $G_n=3^{n-1}G_1\Z^2$ and $H_n=3^{n-1}H_1\Z^2$, and note that $A(G_n)=H_n$, where
\begin{align*}
    G_1=\begin{pmatrix}
        2&0\\
        0&2
    \end{pmatrix}, 
        H_1=\begin{pmatrix}
        4&0\\
        2&1
    \end{pmatrix}, \mbox{ and }A=\begin{pmatrix}
        2&0\\1&1/2
    \end{pmatrix}.
\end{align*}
Theorem \ref{Theorem: coe-iff-structure} implies that $\Z^2_{(G_n)}$ and $\Z^2_{(H_n)}$ are COE.
On the other hand, note that \begin{align*}
    \Z^2/G_n\cong (\Z/2\cdot3^{n-1}\Z)\times (\Z/2\cdot3^{n-1}\Z) \mbox{ and } \Z^2/H_n\cong (\Z/4\cdot 3^{n-1}\Z)\times (\Z/3^{n-1}\Z)
\end{align*}
 Propositions  \ref{Lemma: decomposition-odometers} and \ref{lemma:multidimensional-decomposition-odometers} guarantee that 
 \begin{align*}
     \Z^2_{(G_n)}\cong \Z_3\times \Z_3\times \Z/2\Z\times \Z/2\Z,\mbox{ and }\Z^2_{(H_n)}\cong \Z_3\times\Z_3\times \Z/4\Z.
 \end{align*}
 Consequently, $\Z^2_{(G_n)}$ and $\Z^2_{(H_n)}$ are not isomorphic as abstract groups.
\end{exem}

The issue with the previous example is that the odometers that we considered in the example are not torsion-free groups. 
When this condition is assumed, then we can guarantee that COE odometers are isomorphic as abstract groups.
To this, we use the algebraic description for $\Z^d$ odometers previously presented and  the following characterization of COE for odometers in $\Z^d$ given in \cite{GPS2019}.
We use the notation of Equation \eqref{eq:SuperGroup}.

\begin{theorem}[{\cite[Theorem 1.5]{GPS2019}, \cite[Proposition 2.7]{MerenkovSabitova}}]\label{theo:CharCOEodo} Let  $\Z^{d_G}_{(G_n)}$ and $\Z^{d_H}_{(H_n)}$ be two $\Z^{d_G}$-odometer and $\Z^{d_H}$-odometer associated with scales $(G_n)_{n\geq 0}$ and $(H_n)_{n\geq 0}$, respectively. 
Then these two odometers are COE if and only if,  $d_G = d_H$ and there is a $\psi\in{\rm GL}_{d_G}(\QQ)$, with $\det \psi = \pm 1$ such that $\psi(G^*_\infty) = H^*_\infty$. 
\end{theorem}

Let us mention \cref{theo:CharCOEodo} has been originally proved in \cite{GPS2019} for $d= 1,2$ and in \cite[Proposition 2.7]{MerenkovSabitova} for the general case. 

We are ready to prove:
\begin{theorem}\label{theo:IsoCOEZdOdometer}
    If $(X,\alpha, \Z^d)$ and $(Y,\beta, \Z^d)$ are two COE odometers such that $X$ and $Y$ are torsion-free groups, then $X$ and $Y$ are isomorphic as abstract groups. 
\end{theorem}
\begin{proof}
Let $(G_n)_{n\geq 0}$ and $(H_n)_{n\geq 0}$ be  scales of $\Z^d$ such that $X=\Z^d_{(G_n)}$ and $Y=\Z^d_{(G_n)}$ are torsion-free groups.
Thus, for a prime number $p\in \Z^+$, $p\in \mathcal{P}_{(G_n)}^i$ if and only if $r_{(G_n)}^i(p)=\infty$ and analogously for $(H_n)_{n\geq 0}$.
For a prime number $p\in \Z^+$, we denote by $r_{(G_n)}(p)$ the number of $i$'s in $\{1,\ldots, d\}$ such that $r^i_{(G_n)}(p)=\infty$.
That is, $r_{(G_n)}(p)=\sum_{i=1}^d\mathbbm{1}_{\mathcal{P}_{(G_n)}^i}(p)$.
Similarly, define $r_{(H_n)}(p)=\sum_{i=1}^d\mathbbm{1}_{\mathcal{P}_{(H_n)}^i}(p)$.
Using propositions \ref{lemma:multidimensional-decomposition-odometers} and \ref{Lemma: decomposition-odometers}  they are the following identification  as abstract groups,
\begin{align*}
    X\cong \prod_{\substack{p\; {\rm prime}\\r_{(G_n)}(p)>0}}(\Z_p)^{r_{(G_n)}(p)},\mbox{ and }     Y\cong \prod_{\substack{p\; {\rm prime}\\r_{(H_n)}(p)>0}}(\Z_p)^{r_{(H_n)}(p)}.
\end{align*}
In order to prove that $X\cong Y$ as abstract groups, we guarantee that  $r_{(G_n)}(p)=r_{(H_n)}(p)$ for every prime number $p\in \Z^+$.
Using \cite[Lemma 2.9.4]{RiZa10}, we know that the Pontryagin duals $\hat{X}$ and $\hat{Y}$ of $X$ and $Y$, respectively, are given by 
\begin{align*}
    \hat{X}\cong \bigoplus_{\substack{p\;{\rm prime}\\r_{(G_n)}(p)>0}} (\Z(p^\infty))^{r_{(G_n)}(p)}, \mbox{ and }    \hat{Y}\cong \bigoplus_{\substack{p\;{\rm prime}\\r_{(H_n)}(p)>0}} (\Z(p^\infty))^{r_{(H_n)}(p)}.
\end{align*}
where $\Z(p^\infty)$ is the Prüfer $p$-group and recalling that $\hat{\Z_p}\cong \Z(p^\infty)$ (see, for instance, \cite[Example 2.9.5]{RiZa10}). 
Consider $G^*_\infty$ and $H^*_\infty$ the abelian groups defined in \eqref{eq:SuperGroup} for the scales $(G_n)_{n\geq 0}$ and $(H_n)_{n\geq 0}$, respectively.
Since $X$ and $Y$ are COE, then by Theorem \ref{theo:CharCOEodo} there exists $\psi\in \GL_d(\mathbb{Q})$ such that $\psi(G^*_\infty)=H^*_\infty$.
Observe that by \cite[Theorem 2.6]{GPS2019} we have that $\hat{X}\cong G^*_\infty/\Z^d$ and similarly $\hat{Y}\cong H^*_\infty/\Z^d$. 
As $\psi\in \GL_d(\mathbb{Q})$, we have that $\psi(\Z^d)\leqslant H^*_\infty$ is also a free group in $H^*_\infty$.
Therefore, using \cite[Chapter 12, Page 414]{Fu15}, we deduce that $r_{(H_n)}(p)$ are invariants if we quotient $H^*_\infty$ by any free subgroup of $H^*_\infty$, namely $\Z^d$ or $\psi(\Z^d)$.
Since these numbers are invariants under group isomorphism and $G^*_\infty/\Z^d\cong H^*_\infty/\psi(\Z^d)$,
then, $r_{(G_n)}(p)=r_{(H_n)}(p)$ for any prime integer $p\in \Z^+$.
Thus, we conclude the proof.
\end{proof}

We also have that for $\Z^d$ the transformation groups of odometers that are COE are commensurable. 
To that, we first extend the group $\GL(\Z^d,\alpha)$ using the description given in Proposition \ref{lemma: if-only-Norm}.
Denote by $\SL^{\pm}_d(\mathbb{Q})$ the subgroup of matrices in $M_d(\mathbb{Q})$ with determinant  $1$ or $-1$ ($\pm 1$).  
Without leading to a confusion, for a matrix $A\in \GL_d(\mathbb{Q})$, we also refer to $A$ to the isomorphism induced by it $A:\mathbb{Q}^d\to \mathbb{Q}^d$, and vice versa.

Let $(G_n)_{n\geq 0}$ be a scale and $(\Z^d_{(G_n)},\alpha, \Z^d)$ be its odometer associated.
We define its \textit{transformation group over} $\mathbb{Q}$ by
\begin{align*}
    \GL(\mathbb{Q}^d,\alpha) = \{P\in \SL^{\pm}_d(\mathbb{Q})\mid \forall n>0, \exists m(n)>n  \textrm{ s.t. } P^{-1}(G_m), P(G_{m})\leqslant G_n\}. 
\end{align*}
Note that $\GL(\mathbb{Z}^d,\alpha) = \GL(\mathbb{Q}^d,\alpha)\cap \GL_d(\Z)$.
\begin{lemma}\label{lemma: Q-transformation-group}
    Let $(\Z^d_{(G_n)},\alpha,\Z^d)$ and $(\Z^d_{(H_n)},\beta,\Z^d)$ be two odometers that are continuous orbit equivalent
    with respective scales $(G_n)_n$, $(H_n)_n$ satisfying Theorem \ref{Theorem: coe-iff-structure}, with group isomorphism $\theta:H_0\to G_0$.
    Then, $$\GL(\Z^d,\alpha)=\theta\GL(\mathbb{Q}^d,\beta)\theta^{-1}\cap \GL_d(\Z). $$
\end{lemma}
\begin{proof}

Let $T\in \GL(\Z^d,\alpha)$.
That is, for every $n\geq 0$, there exists $m(n)\geq n$ such that $T(G_m),T^{-1}(G_m)\leq G_n$.
Now, using that $\theta(H_k)=G_k$, for every $k\geq 0$, we have that 
\begin{align*}
    T(\theta (H_m)),T^{-1}(\theta (H_m))\leqslant \theta (H_n).
\end{align*}
Therefore, $\theta^{-1} T \theta\in \GL(\mathbb{Q}^d,\beta)$, and consequently, $T\in \theta\GL(\mathbb{Q}^d,\beta)\theta^{-1}\cap\GL_d(\mathbb{Z})$.
The other inclusion follows in an analogous way.
\end{proof}

\begin{corollary}
        If $(\Z^d_{(G_n)},\alpha, \Z^d)$ and $(\Z^d_{(H_n)},\beta, \Z^d)$ are two continuous orbit equivalent $\Z^d$-odometers, then their transformation groups are commensurable.
\end{corollary}
\begin{proof}
Let assume that $(G_n)_{n\geq 0}$ and $(H_n)_{n\geq 0}$ are two scales satisfying Theorem \ref{Theorem: coe-iff-structure} with the group isomorphism $\theta: H_0\to G_0$. 
As noticed before, since $[\Z^d:G_0]=[\Z^d:H_0]$, $\theta\in \SL_d^{\pm}(\mathbb{Q})$.
Notice that applying conjugation by $\theta$ and Lemma \ref{lemma: Q-transformation-group}, we have that the groups 
\begin{align}\label{eq:finiteindexsubgroups-in-GL} \GL(\mathbb{Q}^d,\beta)\cap \GL_d(\Z)\cap \theta^{-1}\GL_d(\mathbb{Z})\theta &=
 \GL(\Z^d,\beta)\cap \theta^{-1}\GL_d(\Z)\theta \subseteq \GL(\mathbb{Z}^d,\beta)   \\
\nonumber \GL(\mathbb{Q}^d,\alpha)\cap \GL_d(\Z)\cap \theta\GL_d(\mathbb{Z})\theta^{-1}&= \GL(\Z^d,\alpha)\cap \theta\GL_d(\mathbb{Z})\theta^{-1}\subseteq \GL(\mathbb{Z}^d,\alpha)   
\end{align}
 are isomorphic.

Using that for any element $A\in \GL_d(\mathbb{Q})$, the group $A\GL_d(\Z)A^{-1}\cap \GL_d(\Z)$ has a finite index in $\GL_d(\Z)$ (see, for instance, \cite{Ma91}), the subgroup $\theta\GL_d(\Z)\theta^{-1}\cap \GL_d(\Z)$ has a finite index in $\GL_d(\Z)$.
Therefore, $\GL(\Z^d,\beta)\cap \theta^{-1}\GL_d(\Z)\theta$ is a finite index subgroup of $\GL(\Z^d,\beta)$ and $\GL(\Z^d,\alpha)\cap \theta \GL_d(\Z)\theta^{-1}$ is a finite index subgroup of $\GL(\Z^d,\alpha)$ that are isomorphic.
\end{proof}

Using these results, we can also deduce that normalizers are commensurable by COE $\mathbb{Z}^d$- odometers.

\begin{proposition}\label{prop:commensurableNormalizer}
    If $(\Z^d_{(G_n)},\alpha, \Z^d)$ and $(\Z^d_{(H_n)},\beta, \Z^d)$ are two continuous orbit equivalent $\Z^d$-odometers, then their normalizer groups are commensurable.
\end{proposition}
\begin{proof}
        Let assume that $(G_n)_{n\geq 0}$ and $(H_n)_{n\geq 0}$ are two scales for $\Z^d_{(G_n)}$ and $\Z^d_{(H_n)}$, respectively, that satisfy Theorem \ref{Theorem: coe-iff-structure} with $\theta: H_0\to G_0$ the isomorphism in that theorem. To ease notation, we denote by $X=\mathbb{Z}^d_{(G_n)}$ and $Y=\mathbb{Z}^d_{(H_n)}$.
    Denote by $\iota(C_0^X)\leqslant \AutDyn(X,\alpha)$ the subgroup given by 
    \begin{align*}
        \iota(C_0^X)=\{R_{\bm{\xi}}\in \AutDyn(X,\alpha)\mid \bm{\xi}\in C_0^X\}, 
    \end{align*}
    where $R_{\bm{\xi}}$ is the translation by $\bm{\xi}\in X$ given by $R_{\bm{\xi}}(\bm{\zeta})=\bm{\zeta}+\bm{\xi}$, defined in 
    Proposition \ref{prop:right-aut} and $C_0^X$ the $0$-level cylinder set of $X$ defined as in \eqref{eq: cylinders firstlevel}.
    In a similar manner, define $\iota({C_0^Y})$.
    Notice that $C_0^X$ is isomorphic to $\iota(C_0^X)$ and $C_0^Y$ is isomorphic to $\iota(C_0^Y)$.

    Now, define the subgroups of ${G_0^{fix}}\leqslant \GLN(\mathbb{Z}^d,\alpha)$ and ${H_0^{fix}}\leqslant \GLN(\mathbb{Z}^d,\beta)$ as
    \begin{align*}
        {G_0^{fix}}=\GL(\mathbb{Z}^d,\alpha)\cap \theta\GL_d(\mathbb{Z})\theta^{-1}\cap  M(G_0)\GL_d(\mathbb{Z})( M(G_0))^{-1},\\
        {H_0^{fix}}=\GL(\mathbb{Z}^d,\beta)\cap \theta^{-1}\GL_d(\mathbb{Z})\theta\cap  M(H_0)\GL_d(\mathbb{Z})(M(H_0))^{-1},
    \end{align*}
    where $M(H_0), M(G_0)\in M_d(\mathbb{Z})$ denote the $d\times d$ matrices associated to $H_0$ and $G_0$, respectively.
    That is, $H_0=M(H_0)\Z^d$ and $G_0=M(G_0)\Z^d$ and $|\det(M(H_0))|=[\mathbb{Z}^d:H_0]$.\footnote{This matrix is unique up to right multiplication of elements in $\GL_d(\mathbb{Z})$.}
    Using that $M(G_0), M(H_0)\in \GL_d(\mathbb{Q})$ and that $$\GL_d(\mathbb{Q})\subseteq \{A\in \GL_d(\mathbb{R})\mid [\GL_d(\Z):A\GL_d(\Z)A^{-1}\cap \GL_d(\Z)]<\infty\},$$ we deduce that ${G_0^{fix}}$ and ${H_0^{fix}}$ have finite index over $\GLN(\mathbb{Z}^d,\alpha)$ and $\GLN(\mathbb{Z}^d,\beta)$, respectively.
    Moreover, conjugation by $\theta^{-1}$ and the fact that $\theta(H_0)=G_0$ give us that $G_0^{fix}$ and $H_0^{fix}$ are isomorphic.
    
Since $G_0^{fix}\subseteq M(G_0)\GL_d(\Z)M(G_0)^{-1}$, every $\psi\in G_0^{fix}$ satisfies $\psi(G_0)=G_0$.
Likewise, every $w\in H_0^{fix}$ satisfies $w(H_0)=H_0$.
Therefore ${G_0^{fix}}$ and ${H_0^{fix}}$ preserve the level-zero cylinders $C_0^X$ and $C_0^Y$, respectively.
Since $\No(X,\alpha)=\AutDyn(X,\alpha)\rtimes \GLN(G, \alpha)$ and $\No(Y,\beta)=\AutDyn(Y,\beta)\rtimes \GLN(G, \beta)$, it follows that
$$K^X:=\iota(C_0^X)\rtimes G_0^{fix} \mbox{ and } K^Y:=\iota(C_0^Y)\rtimes H_0^{fix}$$
are subgroups of $\No(X,\alpha)$ and $\No(Y,\beta)$, respectively, by using the product given by Proposition \ref{coro:splitsequence}.
Moreover, $[\No(X,\alpha):K^X]<\infty$, $[\No(Y,\beta):K^Y]<\infty$.

To finish the proof, we claim that the map 
\begin{align*}
F:K^Y\to K^X, F(R_{\bm\xi}, \psi) = (R_{\hat{\theta}(\bm\xi)}, {\theta \psi \theta^{-1}}), 
\end{align*} 
is a group isomorphism.  

Indeed, first, observe that it is well defined since 
 $\psi\in H_0^{fix}\subseteq \theta^{-1}\GL_d(\mathbb{Z})\theta$.
The fact that $\hat{\theta}$ is a bijection implies that $F$ is a bijection.

To prove that $F$ is a homomorphism, let $R_{\bm\xi},R_{\bm\zeta}\in \iota(C_0^Y)$ and $\psi,\varphi\in {H_0^{fix}}$.
Note that the multiplication in $N(Y,\beta)$ seen as semidirect product is given by 
\begin{align*}
    (R_\xi,\psi)\cdot (R_\zeta,\varphi)=(R_{\xi}\circ R_{B^Y_\psi(\zeta)},\psi\varphi),
\end{align*}
where $B^Y_\psi$ are the maps given in Proposition \ref{coro:splitsequence} and Remark \ref{rem:desc-B_psi}.
Similarly, we have the same rule for $N(X,\alpha)$ by using the maps $B^X_\omega$ for $\omega \in \GLN(G,\alpha)$.
Using the definitions of $B^Y_\varphi$ and $\hat{\theta}$ we have that 
\begin{align*}
    \hat{\theta}(B_\varphi^Y(\bm\xi))=B_{\theta\varphi\theta^{-1}}^X(\hat{\theta}(\bm\xi)).
\end{align*}
This equation implies that $F$ is a homomorphism.
Therefore, $\No(X,\alpha)$ and $\No(Y,\beta)$ are commensurable.
\end{proof}

When we consider non-abelian groups, the situation is completely different as we see right below.

The following example gives that COE odometers for non-abelian groups can give non-commensurable Normalizers groups and even not quasi-isometric.
\begin{exem}[Continuously orbit equivalent odometers with non-commensurable normalizer groups]\label{ex:NonQINormalizer}
Consider the groups $G^{(A)} = \mathbb{Z}^2\rtimes_\varphi \Z/4\Z$ and $H^{(D)}
=\Z^2\rtimes_\phi \Z/4\Z$, where
\begin{align*}
    \varphi :\Z/4\Z&\to \GL_2(\Z)& \phi:\Z/4\Z&\to \GL_2(\Z)\\
    1&\mapsto A:= \begin{pmatrix}
        0&-1\\
        1&0
    \end{pmatrix}& 
    1&\mapsto D:=\begin{pmatrix}
    -1&0\\
    0&-1
    \end{pmatrix}.
\end{align*} 
For each $n\in \N$, consider the finite index normal subgroups $G_n = 2^n\Z^2\times \{0\}\triangleleft G$ and $H_n = 2^n\Z^2\times \{0\}\triangleleft H$.
The  sequences $(G_n)_{n\geq 0}$ and $(H_n)_{n\geq 0}$ are scales.
It is clear that $G_0\cong H_0$ with the identity map as the group isomorphism, and under this isomorphism $G_n$ is sent to $H_n$.
Moreover, notice that $[G:G_0] = [H:H_0] =4$.
Therefore, Theorem \ref{Theorem: coe-iff-structure} implies that the odometers $(G_{(G_n)},\alpha, G)$ and $(H_{(H_n)},\beta, H)$ are continuously orbit equivalent.

To compare the normalizer groups of $ G_{(G_n)}$ and $H_{(H_n)}$ we first describe $\Aut(G)$ and $\Aut(H)$.

 We begin with $\Aut(G)$: 
Any element in $G$ is of the form $({n},r)$, with ${n}\in \Z^2$ and $r\in \Z/4\Z$.
Moreover, for every ${v}\in \Z^2$ the following relation holds:
\begin{align}\label{eq: 2-G}
({0},1)({v},0) = (A{v}, 0)({0},1).
\end{align}
Notice that $G_0$ is the only  torsion free normal subgroup of minimal index $4$ of $G$.
Therefore, any isomorphism $\psi\in \Aut(G)$ preserves this subgroup $\psi(G_0) = G_0$.
So, every $\psi\in \Aut(G)$ is characterized by $r\in (\Z/4\Z)^*$, ${w}\in \Z^2$ and $R\in \GL_2(\Z)$ such that
\begin{align*}
    \psi({0},1) = ({w},r),\;
    \psi({v},0) = (R{v},0), \mbox{ for any ${v}\in \Z^2$.}
\end{align*}
Actually, the linear map  $R$ has some restrictions.
Applying $\psi$ on \eqref{eq: 2-G} provides  $    ({w},r)(R{v},0) = (RA{v},0)({w},r)$,
and consequently, $({w}+A^rR{v},r) = ({w}+RA{v},r)$.
This implies that $R$ belongs to the subgroup $\No_{\GL_2(\Z)}(\langle A\rangle)$, the normalizer of the group $\langle A\rangle$ in $\GL_2(\Z)$.
A standard calculation shows that $\No_{\GL_2(\Z)}(\langle A\rangle) = \langle A, C\rangle$ is isomorphic to the dihedral group,  where $C = \begin{pmatrix}
    -1&0\\0&1
\end{pmatrix}$.
More precisely, it is possible to check that 
\begin{align*}
    \psi({0},1) = ({w},r),\;
    \psi({v},0) = (R{v},0)      \mbox{ for any } {v}\in \Z^2,
\end{align*}
where
\begin{align*}
\begin{cases}
        R\in \langle A\rangle , &\mbox{ if }r=1\\
        R\in \No_{\GL_2(\Z)}(\langle A\rangle)\setminus\langle A\rangle, &\mbox{ if }r=3.
    \end{cases}
\end{align*}

This description implies that $\Aut(G)$ is a semidirect product of the normal subgroup $N$ and the subgroup $M$, where
\begin{align*}
    N =& \{\psi\in \Aut(G) : \psi({v},0) = ({v},0), \forall{v}\in \Z^2  \},\\
    M =& \{\psi\in \Aut(G) : \psi({0},1) = ({0},1) \}\cup \{\psi\in \Aut(G) : \psi({0},1) = ({0},3) \}. 
\end{align*}
In fact, observe that $\psi \in \Aut(G)$, with $\psi(0,1)=(w,r)$, $\psi(v,0)=(Rv,0)$, can be decomposed as $\psi=\psi_R\circ \psi_w$, where $\psi_R(0,1)=(0,r)$, $\psi_R(v,0)=(Rv,0)$ , and $\psi_w(0,1)=(R^{-1}w,1)$, $\psi_w(v,0)=(v,0)$.
It is clear that $\psi_R\in M$ and $\psi_w\in N$.
Moreover, we have that $M\cap N=\{{\rm Id}_G\}$.

On the other hand, $M=\langle \psi_A,\psi_C\rangle$, where $\psi_A(0,1)=(0,1),\psi_A(v,0)=(Av,0)$ and $\psi_C(0,1)=(0,3),\psi_C(v,0)=(Cv,0)$.
Moreover, $\psi_A^4=\psi_C^2=(\psi_A\circ\psi_C)^2=id_G$, hence, $M\cong D_4$.

Observe that for every $i\geq 0$ and $\psi\in \Aut(G)$, $\psi(G_i) = G_i$.
Thus, Lemma \ref{lemma: if-only-Norm} implies that $\GLN(G,\alpha) = \Aut(G)$, and $\No( G_{(G_n)},\alpha)$ is a semidirect product between $ G_{(G_n)}$ and $\Aut(G)$.

Let us describe $\Aut(H)$: 
Analogous to the case with $G$,  $H_0$ is the only free torsion normal subgroup of minimal index $4$.
Therefore, $\psi(H_0) = H_0$ for every element $\psi\in \Aut(H)$.
Again, we have that every $\psi\in \Aut(H)$ is characterized by $R\in \GL_2(\Z)$, $r\in (\Z/4\Z)^*$ and ${w}\in \Z^2$ such that
\begin{align*}
    \psi({0},1)=({w},r),\psi({v},0)=(R{v},0), \mbox{ for every }{v}\in \Z^2.
\end{align*}
The additional relation given to the presentation of $H$ by the semidirect product reads as
\begin{align*}
    ({0},1)({v},0) = (B{v},0)({0},1),\mbox{ for each }{v}\in \Z^2.
\end{align*}
Since $B$ commutes with any element in $\GL_2(\Z)$, there is no  restriction on the choice of $R$.
Thus, $\Aut(H)$ is a semidirect product of the normal subgroup $L$ and the subgroup $K$ given by
\begin{align*}
L =& \{\psi\in \Aut(H): \psi({v},0) = ({v},0), \forall {v}\in \Z^2\}\cong \Z^2\times \Z/2\Z,\\
    K =& \{\psi\in \Aut(H): \psi({0},1) = ({0},1)\}\cong \GL_2 (\Z). 
\end{align*}
This implies that $\Aut(G)$ and $\Aut(H)$ are not commensurable since, for example, $\Aut(G)$ is amenable, whereas $\Aut(H)$ is not.
Moreover, any $\psi\in \Aut(H)$ satisfies  $\psi(H_i) = H_i$ for each $i\geq 0$.
Consequently, $\GLN(H,\beta) = \Aut(H)$ and by Proposition \ref{coro:splitsequence}, $\No(H_{(H_i)},\beta)$ is a semidirect product of $H_{(H_i)}$ by $\Aut(H)$. 
Furthermore, we have that $\No( G_{(G_n)},\alpha)$ and $\No(H_{(H_i)},\beta)$ are not commensurable for the same reason: one is amenable, and the other is not.
\end{exem}

\section{Control on the automorphsim group of Toeplitz subshifts}\label{Sec:Construction-Toeplitz}

In this section, we use the machinery of 
semi-cocycle to give explicit  construction of Toeplitz subshifts with   controlled automorphism groups. The control of this group  is given through a double inclusion in Theorem \ref{Theorem: diverse-Aut}. It is due to the properties of the discontinuity point of the semi-cocycle. In particular, it will enable us in the next section to provide  examples of minimal subshifts with all the acting group (even in the non-abelian case) or only its center as automorphism groups.

Let $(G_n)_{n\geq 0}$ be a scale and $(G_{(G_n)},\alpha, G)$ its odometer associated.
In this section, $G_{(G_n)}$ is endowed with the metric $\rho$ defined in \eqref{eq: metric-odometer}.
Recall that this metric satisfies $\rho(\bm{\xi}_1\bm{\zeta},\bm{\xi}_2\bm{\zeta})=\rho(\bm{\zeta}\bm{\xi}_1,\bm{\zeta}\bm{\xi}_2)= \rho(\bm{\xi}_1,\bm{\xi}_2)$ for any $\bm{\xi}_1,\bm{\xi}_2,\bm{\zeta}\in G_{(G_n)}$.
For every $\bm{\xi}\in  G_{(G_n)}$, define the set $A_{\bm{\xi}}$ as follows:
$$A_{\bm{\xi}}=\{{g}\in{G}: \bm{\xi}\tilde{g}\bm{\xi}^{-1}\notin\tilde{G}\}.$$  That is, $A_{\bm\xi}$ is the set that satisfies $\tilde{A_{\bm\xi}}=\tilde{G}\setminus [\bm\xi\tilde{G}\bm\xi^{-1}\cap \tilde{G}]$, where\; $\tilde{}:G\to G_{(G_n)}$ is the embedding map introduced in Section \ref{Subsection: odometers}.
From this, we notice that $A_{\bm\xi}^c:=G/A_{\bm\xi}$ is the set such that $\tilde{A_{\bm\xi}^c}=\bm\xi\tilde{G}\bm\xi^{-1}\cap \tilde{G}$, which is a group.
Hence, $A_{\bm\xi}^c$ is itself a group.

\begin{theorem}\label{Theorem: diverse-Aut}
    Let $G$ be a countable, residually finite, infinite group, $(G_n)_{n\geq 0}$  a scale and $\bm{\zeta}_0,\bm{\zeta}_1,\dots, \bm{\zeta}_s\in G_{(G_n)}$ be such that $\bm{\zeta}_0=\tilde{1_G}$ and $O_\alpha(\bm{\zeta}_i)\cap O_\alpha(\bm{\zeta}_j)=\emptyset$ whether  $0\leq i<j\leq s$.
     Then, there exist
     $g_1,\dots, g_s\in G$  
and  a Toeplitz subshift  $(X,\sigma, G)$, whose maximal equicontinuous factor is $G_{(G_n)}$, and its automorphism group satisfies
    \begin{align*}
       (Z(G) \leqslant)\quad\   \Centr_{G}(\Theta)\leqslant \AutDyn(X,\sigma)\leqslant \bigcap_{i=1}^s A_{\bm\zeta_i}^c \quad\ (\leqslant G),
            \end{align*}
    where $\Centr_G(\Theta) = \{g\in G: \tilde{g}\bm{\zeta}=\bm{\zeta}\tilde{g}, \forall \bm{\zeta}\in \Theta\}$, and $\Theta=\{\tilde{g_j}\bm{\zeta}_i : 0 \leq i \leq s, 0\leq j\leq s-i\}$ with $\tilde{g_0} = \tilde{1_G}$.
\end{theorem}
Now, we present a proof of Theorem \ref{Theorem: diverse-Aut}.
We first proceed with the construction of a Toeplitz array via semicocycles:
Let $(G_n)_{n\geq 0}$ be a scale and $\bm{\zeta}_0,\dots, \bm{\zeta}_s\in G_{(G_n)}$ be as in Theorem \ref{Theorem: diverse-Aut}.
Let $n_1,\dots, n_s\in \N$, $g_0,\dots, g_s\in G$ be
such that $n_t+2<n_{t+1}$, and  $\frac{1}{2^{n_t}}<\frac{1}{5}\min\left\{{\rho(\bm{\zeta}_i,\bm{\zeta}_j)}: 0\leq i<j\leq s\right\}$, with
   $g_t\in G_{{n_{t-1}}}\setminus G_{{n_t}}$, for all $t\in \{1,\dots, s\}$.
Denote $
    r':=\frac{1}{5}\min\left\{{\rho(\bm{\zeta}_i,\bm{\zeta}_j)}: 0\leq i<j\leq s\right\}.$
Observe that $\tilde{g_j}\bm{\zeta}_i\neq \tilde{g_k}\bm{\zeta}_t$ whenever $j\neq k$ or $i\neq t$ for $0\leq i,t\leq s$, $0\leq j\leq s-i$, and $0\leq k\leq s-t$.
Indeed, if $i\neq t$,
\begin{align*}
    \rho(\tilde{g_j}\bm{\zeta}_i,\tilde{g_k}\bm{\zeta}_t)&\geq \rho(\tilde{g_j}\bm{\zeta}_i,\tilde{g_j}\bm{\zeta}_t)-\rho(\tilde{g_k}\bm{\zeta}_t,\tilde{g_j}\bm{\zeta}_t)\\
    &= \rho(\bm{\zeta}_i,\bm{\zeta}_t)-\rho(\tilde{g_k},\tilde{g_j})>0.
\end{align*}
If $i = t$ and $j\neq k$, then $\rho(\tilde{g_j}\bm{\zeta}_i,\tilde{g_k}\bm{\zeta}_i) = \rho(\tilde{g_j},\tilde{g_k})>0$.

Denote by $m:=\max\{-\log_2(\rho(\bm{\zeta},\bm{\zeta}'))+1: \bm{\zeta},\bm{\zeta}'\in\Theta\mbox{ are distinct elements}\}<\infty$.
This, in particular,  implies that $\frac{1}{2^{m-1}}\leq \rho(\bm\zeta,\bm\zeta')$ for every pair of distinct elements $\bm\zeta,\bm\zeta'\in \Theta$.
For every $\bm\zeta\in \Theta$, consider a primer number $p_{\bm\zeta}\in \Z^+$ in such a way $p_{\bm\zeta}\neq p_{\bm\zeta'}$ for distinct elements $\bm\zeta,\bm\zeta'\in \Theta$.
For every $\bm\zeta\in \Theta$, consider the map 
\begin{align}\label{eq:f_p}
f'_{\bm\zeta}:\left\{\frac{1}{2^n}:n\in\mathbb{N}_0\right\}\to \{0,1\},\;    f'_{\bm\zeta}\left(\frac{1}{2^n}\right)=\begin{cases}
        0,& \mbox{if }n\equiv 0\bmod{p_{\bm\zeta}} \mbox{ or }n\leq m,\\
        1,&\mbox{otherwise.}
    \end{cases}
\end{align}
Define the map 
\begin{align}\label{eq:trivial-aut}
f: G_{(G_n)}\to \Z,\;    f(\bm{\omega})=\begin{cases}
        0,&\mbox{ if }\bm{\omega}\in\Theta,\\
        \sum_{\bm{\zeta}\in\Theta} f'_{\bm\zeta}(\rho(\bm{\omega},\bm{\zeta})),& \mbox{ otherwise.}
    \end{cases}
\end{align}
Observe that this map only take values in $\{0,1\}$. 
In fact, if there exists $\bm{\omega}\in  G_{(G_n)}$ such that $\sum_{\bm\zeta\in \Theta} f'_{\bm\zeta}(\rho(\bm\omega, \bm\zeta))>1$. 
In particular, there exist different elements $\bm\zeta,\bm\zeta'\in \Theta$ such that $\rho(\bm{\omega},\bm{\zeta})=\frac{1}{2^a}$, and $\rho(\bm{\omega},\bm{\zeta}')=\frac{1}{2^b}$ for integers $a,b\in \mathbb{N}$ with $a,b>m$. Then, 
\begin{align*}
    \frac{1}{2^{m-1}}\leq\rho(\bm{\zeta},\bm{\zeta}')< \frac{1}{2^a}+\frac{1}{2^b}<\frac{1}{2^{m-1}},
\end{align*}
a contradiction with the choice of $m$.
The next lemma shows it is possible to select the orbits  of $\bf \zeta\in \Theta$ from properties of their $\pi$-fibers. 
\begin{lemma}\label{prop:second-Toeplitz}
The map  $f$ defined in \eqref{eq:trivial-aut} is an invariant under no rotations semicocycle over $( G_{(G_n)},\bm{\xi})$, for every $\bm{\xi}\in  G_{(G_n)}\setminus \bigcup_{\bm{\zeta}\in \Theta}O_\alpha(\bm{\zeta})$, and the system $(X_f,\sigma, G)$ is a Toeplitz extension of $ (G_{(G_n)},\alpha, G)$, whose maximal equicontinuous factor is equal to $ G_{(G_n)}$.
Moreover,
\begin{align*}
    |\pi^{-1}(\bm\zeta)|=\begin{cases}
        2^{s+1-i},&\mbox{ if }\bm\zeta\in O_\alpha(\bm\zeta_i), 0\leq i\leq s,\\
        1,&\mbox{ otherwise.}
    \end{cases}
\end{align*}
\end{lemma}
\begin{proof}
 We claim that the set $D_f$ of points of discontinuity of $f$ is equal to $\Theta$.
In fact, let $\bm{\zeta}\in \Theta$ and let $(\gamma_n)_{n\geq 0}\subseteq G$ be a sequence in such a way {$\gamma_{2n}\in G_{p_{\bm\zeta}n}\setminus G_{p_{\bm\zeta}n+1}$} and $\gamma_{2n+1}\in G_{p_{\bm\zeta}n+1}\setminus G_{p_{\bm\zeta}n+2}$ for each $n\in\mathbb{N}$. 
Note that $\tilde\gamma_n\bm{\zeta}\to \bm{\zeta}$ as $n\to\infty$ since $\rho(\tilde \gamma_{2n}\bm{\zeta},\bm{\zeta})=\frac{1}{2^{p_{\bm\zeta}n}}$ and $\rho(\tilde \gamma_{2n+1}\bm\zeta,\bm\zeta)=\frac{1}{2^{p_{\bm\zeta}n+1}}$.
Moreover, note that for every $n>\frac{m}{p_{\bm\zeta}}$, we have that $f(\tilde\gamma_{2n})=0$ and $f(\tilde\gamma_{2n+1})=1$.
Taking the sequences $(\gamma_{2n+1})_{n\geq 0}$ and $(\gamma_{2n})_{n\geq 0}$ we obtain the discontinuity of $\bm{\zeta}$.
Therefore, $\Theta\subseteq D_f$. 
The other containment is a direct consequence of the ultrametric $\rho$: any triple of points forms an isosceles triangle.

To show that $f$ is invariant under no rotations we will use \cite[Proposition 2.3]{FuKw19}.
Let $\bm{\zeta},\bm{\eta}\in  G_{(G_n)}$ be such that $\bm{\zeta}\neq \bm{\eta}$.
If $\bm{\zeta}^{-1}\bm{\eta}\notin D_f$, we conclude, since $\bm{\zeta}^{-1}\bm{\zeta}\in D_f$.
Assume that $\bm{\zeta}^{-1} \bm{\eta}\in D_f$. 
Then, there exist $0\leq i\leq s $ and $0\leq t\leq s-i$ satisfying  $\bm{\zeta}^{-1} \bm{\eta}=g_t\bm{\zeta}_i$.
Observe that $\tilde{g_s}\bm{\zeta}^{-1}\bm{\zeta}=\tilde{g_s}\in D_f$.
We claim that $\tilde{g_s}\bm{\zeta}^{-1}\bm{\eta}\notin \Theta$.
Indeed, let $\tilde{g_w}\bm{\zeta}_j\in \Theta$.
\begin{itemize}
    \item If $j\neq i$, then
\begin{align*}
  \rho(\tilde{g_s}\tilde{g_t}\bm{\zeta}_i,\tilde{g_w} \bm{\zeta}_j)&\geq \rho(\tilde{g_s}\tilde{g_t}\bm{\zeta}_i, \tilde{g_s}\tilde{g_t}\bm{\zeta}_j)-\rho(\tilde{g_s}\tilde{g_t}\bm{\zeta}_j, \tilde{g_s}\tilde{g_w}\bm{\zeta}_j)-\rho(\tilde{g_s}\tilde{g_w}\bm{\zeta}_j, \tilde{g_w}\bm{\zeta}_j)\\
  &=\rho(\bm{\zeta}_i,\bm{\zeta}_j)-\rho(\tilde{g_t},\tilde{g_w})-\rho(\tilde{g_s},\tilde{1_G})\\
  &>4r'-2r'-r'=r'.
\end{align*}
\item Assume $j=i$ and $w> t$ (the case $w< t$ is analogous).
Note that
\begin{align*}
\rho(\tilde{g_s}\tilde{g_t}\bm{\zeta}_i,\tilde{g_w}\bm{\zeta}_i)=\rho(\tilde{g_s}\tilde{g_t},\tilde{g_w})&\geq \rho(\tilde{g_s}\tilde{g_t},\tilde{g_s}\tilde{g_w})- \rho(\tilde{g_s}\tilde{g_w},\tilde{g_w})\\
    &=\rho(\tilde{g_t},\tilde{g_w})-\rho(\tilde{g_s},\tilde{1_G})\\
    &\geq \rho(\tilde{g_t},\tilde{1_G})-\rho(\tilde{g_w},\tilde{1_G})-\rho(\tilde{g_s},\tilde{1_G})\\
    &>\frac{1}{2^{n_t+1}}-\frac{1}{2^{n_w}}-\frac{1}{2^{n_s}}\\
    &\geq \frac{1}{2^{n_t+1}}-\frac{1}{2^{n_s-1}}>0.
\end{align*}
\item If $j=i$ and $w=t$, then $\rho(\tilde{g_s}\tilde{g_t}\bm{\zeta}_i,\tilde{g_t}\bm{\zeta}_i)=\rho(\tilde{g_s},\tilde{1_G})>0$.

\end{itemize}
Therefore,  $\tilde{g_s}\bm{\zeta}^{-1} \bm{\eta}=\tilde{g_s} \tilde{g_t}\bm{\zeta}_i\notin \Theta$.
Consequently, $f$ is invariant under no rotations.

That $(X_f,\sigma, G)$ is a Toeplitz subshift whose maximal equicontinuous factor is equal to $G_{(G_n)}$ follows from Lemma \ref{lemma:invariantnorotation}.

 To compute the cardinality of the fiber, since it is constant along an orbit, it is enough to compute the one of $\bm\zeta_i \quad (\in \Theta)$ for $0\leq i\leq s$. 
Subsection \ref{Subsec: Semicocycles and Toeplitz} gives that $X_f=\overline{O_\alpha(x_{f,\xi})}$ for any $\bm\xi\notin D_f$.
Recall that for an element $x\in \pi^{-1}(\bm\zeta_i)$, we have that $x(g)\in F(\tilde g\bm\zeta_i)=\{k\in \{0,1\}: (\tilde g\bm\zeta_i,k)\in F\}$, where $F$ is the closure of the graph of $f$.
In particular, every $x,x'\in \pi^{-1}(\bm\zeta_i)$ satisfy  $x(g)=x'(g)$ for any $g\in G$ such that $\tilde g\bm\zeta_i\notin D_f$.
Thus, the elements in $\pi^{-1}(\bm\zeta_i)$ can only differ at  the positions $g_0, g_1,\dots, g_{s-i}$ since $g_j\bm\zeta_i\in D_f$ for $0\leq j\leq s-i$.
Conversely, denoting   $S_i=\{g_0,g_1,\dots, g_{s-i}\}$, we claim that for every $Q\in \{0,1\}^{S_i}$, there exists $y\in \pi^{-1}(\bm\zeta_i)$ such that $y(g)=Q(g)$ for every $g\in S_i$.
This  will ensure that $|\pi^{-1}(\zeta_i)|=|\{0,1\}^{S_i}|=2^{s+1-i}$ and prove the theorem. 
Indeed, let $Q\in \{0,1\}^{S_i}$.
Set $m_i$ to be the product of all the primes assigned to the elements of the form $\tilde g\bm\zeta_i$ and by $r_i$ be the product of all primes assigned to the elements of the form $\tilde g\bm\zeta_i$ such that $Q(g)=0$, $g\in S_i$.
That is,
\begin{align*}
    m_i=\prod_{\substack{g\in S_i}}p_{\tilde g\bm\zeta_i}, \mbox{ and }r_i=\prod_{\substack{g\in S_i\\Q(g)=0}}p_{\tilde g\bm\zeta_i}.
\end{align*}
The very definitions of $m_i$ and $r_i$ ensure that $r_i+m_i n = Q(g) \mod p_{\tilde g\zeta_i} $ for any $g \in S_i$.
Let $(\gamma_n)_{n\geq 0}$ be a sequence of elements in $G$ such that $\gamma_n\in G_{r_i+m_in}\setminus G_{r_i+m_in+1}$.
Let $(u_n)_{n\ge0}$, $(v_n)_{n\ge 0}$ and $(h_{n})_{n\geq 0}$ be  sequences of elements of  $G$ such that $(u_nG_n)_n=\bm\xi$, $(v_nG_n)_{n}=\bm\zeta_i$, and  $(u_{r_i+m_in+1})^{-1}h_{n}v_{r_i+m_in+1}\in G_{r_i+m_in+1}$.
We claim that the sequence $(\sigma^{\gamma_nh_n}x_{f,\bm\xi})_{n\geq 0}$ has a limit point $x$ such that $x(g)=Q(g)$ for every $g\in S_i$.
In fact,
for every $g\in S_i$, observe that $(\sigma^{\gamma_nh_n}x_{f,\bm\xi})(g)=x_{f,\bm\xi}(g\gamma_nh_n)=f(\tilde g\tilde \gamma_n\tilde h_n\bm\xi)$.
Note that $v_{r_i+m_in}^{-1}\gamma_nh_nu_{r_i+m_in}\in G_{r_i+m_in}$.
On the other hand, $v_{r_i+m_in+1}^{-1}\gamma_nh_nu_{r_i+m_in+1}\notin G_{r_i+m_in+1}$ since $\gamma_n\notin G_{r_i+m_in+1}$.
Consequently, $\rho(\tilde g\tilde \gamma_n\tilde h_n\bm\xi, \tilde g\bm\zeta_i)=\frac{1}{2^{r_i+m_in}}$.
This in turns implies that $(\sigma^{\gamma_nh_n}x_{f,\bm\xi})(g)=f(\tilde g\tilde \gamma_n\tilde h_n\bm\xi)=f'_{\tilde g\bm\zeta_i}(\tilde g\tilde \gamma_n\tilde h_n\bm\xi)=Q(g)$ for every large enough $n$ (so that $r_i+m_in\geq m$).
Hence any accumulation point $y \in X_f$ of 
$\{\sigma^{\gamma_{n}h_{n}}x_{f,\bm\xi}\}_{n\geq 0}$ satisfies  $y(g)=Q(g)$ for every $g\in S_i$. Moreover, since $\tilde{\gamma}_n \to 1_{G_{(G_n)}}$ and  $\tilde{h_n}\xi\to \zeta_i$ when $n\to\infty$, it is standard to check that $y\in \pi^{-1}(\zeta_i)$. This shows the claim.

\end{proof}

\begin{proof}[Proof of Theorem \ref{Theorem: diverse-Aut}]
By Lemma  \ref{prop:second-Toeplitz}, we have a Toeplitz subshift  $(X_f,\sigma, G)$ so that $(G_{(G_n)},\alpha, G)$ is its maximal equicontinuous factor  associated with a factor map $\pi:X_f\to G_{(G_n)}$.
   Corollary \ref{Coro:injectiveAut} implies that $\pi$ is compatible with $\AutDyn(X_f,\sigma)$ and that $\hat{\pi}:\AutDyn(X_f,\sigma)\to \AutDyn( G_{(G_n)},\alpha)$ is injective. 
   Proposition \ref{prop:right-aut} guarantees that there exists $\bm{\omega}\in G_{(G_n)}$ satisfying $\hat{\pi}(\phi)=R_{\bm{\omega}}$, where $R_{\bm{\omega}}$ is the right multiplication automorphism of $G_{(G_n)}$ defined in Proposition \ref{prop:right-aut}.
    Using Lemma \ref{lemma: fibers} and the fact that $O_\alpha(\tilde{1_G})$ contains all the elements in $G_{(G_n)}$ with a fiber of size $2^{s+1}$ (Lemma \ref{prop:second-Toeplitz}), we obtain 
    $R_{\bm{ \omega}}$ preserves the the orbit $O_\alpha(\tilde{1_G})$ and so there is  a $g \in G$ such that $\tilde{g}=\bm{\omega}$.
        
    Now, assume that $g\in A_{\bm{\zeta}_i}$, for some $1\leq i\leq s$.
    Therefore,  $R_{\tilde{g}}(\bm{\zeta}_i)=\bm{\zeta}_i\tilde{g}\notin O_\alpha(\bm{\zeta}_i)$, which is a contradiction with the last part of Lemma \ref{prop:second-Toeplitz} and Lemma \ref{lemma: fibers}, since the only elements $\bm{\zeta}\in G_{(G_n)}$ such that $|\pi^{-1}(\{\bm{\zeta}\})| = |\pi^{-1}(\{\bm{\zeta}_i\})|$ are those that belong to $O_\alpha(\bm{\zeta}_i)$.
    Consequently, $\AutDyn(X_f,G)\leqslant  \bigcap_{i=1}^s A_{\bm{\zeta}_i}^c$.

On the other hand, if $h\in \Centr_G(\Theta)$, then $\rho(\tilde h^{-1}\tilde s\bm{\xi},\bm{\zeta}) = \rho(\tilde s\bm{\xi} \tilde h^{-1},\bm{\zeta})$ for every $\bm{\zeta}\in \Theta$ and  ${\bm \xi} \in G_{(G_n)}\setminus \cup_{{\bm \zeta} \in \Theta} O_\alpha(\zeta)$.
This implies that $\sigma_{{\rm L}}^h (x_{f,\bm\xi})(s) = x_{f,\bm\xi \tilde h^{-1}}(s)$, where $\sigma_{{\rm L}}^h:\{0,1\}^G\to \{0,1\}^G$ is the left-shift by $h$ defined by $(\sigma_{{\rm L}}^hx)(g)=x(h^{-1}g)$, $g\in G$. 
Therefore, using Lemma \ref{lemma: equal-cocycles}, we have that $\sigma_{{\rm L}}^h|_{X_f}:X_f\to X_f$ is a self-homeomorphism that commutes with $\sigma^g$ for every $g\in G$.
Consequently, $\sigma_{{\rm L}}^h|_{X_f}\in \AutDyn(X_f,\sigma)$. 
Moreover the restricted left-shift is not trivial  $\sigma_{{\rm L}}^h|_{X_f} \neq {\rm Id}$ when $h\neq 1_G$. Otherwise,  since $\bm \xi =\pi(x_{f, {\bm \xi}})$ for any $\bm\xi \in G_{(G_n)}\setminus \cup_{{\bm \zeta} \in \Theta} O_\alpha(\zeta)$ (Subsection \ref{Subsec: Semicocycles and Toeplitz}), the equality  $\sigma_{{\rm L}}^h (x_{f,\bm\xi}) = x_{f,\bm\xi \tilde h^{-1}}$ implies that $\bm \xi = \bm \xi \tilde h^{-1}$  and then $h= 1_G$.
The inclusion $\Centr_G(\theta)\leqslant\AutDyn(X_f,\sigma)$ follows. 
\end{proof}

\section{Applications}\label{Sec:Applications}
In this section we present two applications of Theorem \ref{Theorem: diverse-Aut}.
We provide  explicit constructions of Toeplitz subshifts with a control on their automorphism groups. 
In the first class of examples, their automorphism group is isomorphic to the acting group, even if the group is non-abelian. It applies to linear groups such as the Heisenberg group or the projective linear group (see Examples \ref{example: Heisenberg}, \ref{ex:PSL2}).
 In contrast, their  normalizer groups are the largest possible for a fixed automorphism group. They are extensions of the group of all the isomorphisms of the group by their automorphism groups, whereas the examples have low complexity (Corollary \ref{coro: splitsequence of Xf- Aut} and Proposition \ref{prop:complexity}).  

In the second class of examples, the automorphism groups are trivial (Propositions \ref{Coro: Hei-trivial Toep} and \ref{prop:TrivialCentPSL} for the Heisenberg group and  projective linear groups respectivley). They are restricted to the center of the group. When the group has a trivial center, its automorphism group is also trivial, i.e. contains only the identity transformation. The constructions apply to amenable and non-amenable groups.

\subsection{$G$-Toeplitz subshifts with automorphism group isomorphic to $G$}\label{sec:ToeplAutoG}

\begin{corollary}\label{Coro: s=0}
Let $G$ be a countable, residually finite, infinite group and $(G_n)_{n\geq 0}$ a scale.
Then, there exists a Toeplitz subshift $(X,\sigma,G)$, whose maximal equicontinuous factor is $G_{(G_n)}$ and its automorphism group $\AutDyn(X,\sigma)\cong G$.
\end{corollary}

\begin{proof}
    Apply Theorem \ref{Theorem: diverse-Aut} with $s=0$. In this case, note that $C_G(\Theta)=G$ and as by standard convention we have that $\bigcap_{i\in \emptyset} A_{\zeta_i}^c=G$.
\end{proof}

From now on in this subsection, we denote $(X_f,\sigma_{X_f},G)$  the Toeplitz subshift given by the previous Corollary. We will show it has linear symbolic complexity and we will describe its normalizer according to the scale $(G_n)_{n}$.

\textbf{The subshift $(X_f,\sigma_{X_f},G)$  has low complexity}.
For a subshift $X\subseteq \Sigma^G$  and  $F\subseteq G$ a finite subset, let  us define $L_F(y)$ for $y \in X$ to be the set of colorations of $F$ within $y$:
\begin{align*}
    L_F(y)=\{\sigma^gy|_F\in \Sigma^F\colon g\in G\}.
\end{align*}
Let us denote by $L_F(X)=\bigcup_{y\in Y}L_F(y)$.
When the subshift $X$ is minimal, we have that $L_F(X)=L_F(y)$ for any $y\in X$.

For the subshift $(X_f,\sigma_{X_f},G)$ previously given, let $x\in \pi^{-1}(\tilde{1_G})$ be an element in the fiber of $\tilde{1_G}$ of the factor map $\pi:X_f\to G_{(G_n)}$.
Observe that $x(g)=f(\tilde{g})$ for every $g\neq 1_G$.
Let $(F_n)_{n\geq 0}$ be a sequence of finite sets of $G$ such that every $F_n$ is a fundamental domain of $G_n$, i.e., for every $n\geq 0$, contains exactly one representative of every coset in $G/G_n$.
\begin{lemma}
For every $g\in G$, we have that $x(tg)=x(td)$ for any $t\in F_n\setminus G_nd^{-1}$, where $d\in F_n$ is such that $gG_n=dG_n$.
\end{lemma}
\begin{proof}
Let $g\in G$.
Consider $d\in F_n$ and $\gamma\in G_n$ such that $g=d\gamma$.
Now, set $t\in F_n\setminus G_nd^{-1}$.
Observe that $td\in G_k\setminus G_{k+1}$ for some $0\leq k\leq n-1$.
This follows from $G=\bigcup_{k\geq 0}G_k\setminus G_{k+1}$  and $td\notin G_n$. 
Moreover, note that for every $0\leq k \leq n-1$, as the subgroups $G_i$'s are normal, 
\begin{align*}
    tg\in G_k\setminus G_{k+1}\iff td\in G_k\setminus G_{k+1}.
\end{align*}
This implies that $x(tg)=x(td)$ for every $t\in F_n\setminus G_nd^{-1}$ as desired.
\end{proof}
\begin{proposition}\label{prop:complexity}
    For every $n\geq 0$ it holds 
    \begin{align*}
         |L_{F_n}(X_f)|\leq 2|F_n|.
    \end{align*}
\end{proposition}
\begin{proof}
    Let $n\geq 0$.
 As mentioned previously,     $L_{F_n}(X)=L_{F_n}(x)$ since $X_f$ is minimal. 
    The previous lemma guaranties that for every $g\in G$,  $\sigma^g(x)|_{F_{n}\setminus G_nd^{-1}}=\sigma^d(x)|_{F_n\setminus G_nd^{-1}}$, where $d\in F_n$ is such that $gG_n=dG_n$.
    Therefore, for every $d\in F_n$ we have at most two different patterns in the collection $\{\sigma^{d\gamma}(x)|_{F_n}\colon \gamma\in G_n\}$, since $|F_n\cap G_nd^{-1}|=|F_nd\cap G_n|=1$.
    Thus, we conclude the proof.
\end{proof}

\textbf{Description of the normalizer of  $(X_f,\sigma_{X_f},G)$}.
Observe  that  for any $\psi\in \GLN(G,\alpha)$, Proposition \ref{lemma: if-only-Norm} guarantees that $(\psi(G_n))_{n\geq 0}$ is a scale associated with $ G_{(G_n)}$.
We can then  define a metric $\rho_\psi$ over $ G_{(G_n)}$, that generates the topology of $ G_{(G_n)}$, in the following way
\begin{align*}
    \rho_\psi(({\zeta}_n\psi(G_n))_{n\geq 0},(\zeta_n\psi(G_n))_{n\geq 0})=2^{-k}, \mbox{ where }k=\sup\{i\geq 0: \eta_i\psi(G_i)=\zeta_i\psi(G_i)\}.
\end{align*}

Using this metric, we define the semicocycle $f_\psi: G_{(G_n)}\to \{0,1\}$ as follows: For $\bm{\omega}\in  G_{(G_n)}$,
\begin{align}\label{eq: good-f-hat}
    f_\psi(\bm{\omega})=\begin{cases}
        1,&\mbox{ if }\rho_\psi(\bm{\omega},\tilde{1_G})=2^{-i}, \mbox{ with }i\equiv 0 \bmod{2}\\
        0,&\mbox{otherwise.}
    \end{cases}
\end{align}
Notice that $f_{{\rm Id}_G}$ is the semicocycle $f$ given in \eqref{eq:trivial-aut} with the hypothesis $s=0$ (as in  Corollary \ref{Coro: s=0}), $\bm\zeta_0 =\tilde 1_G$  and $p_{\bm\zeta_0}=2$. In this case, we denote $f_{{\rm Id}_G}$ just as $f$.

For $\psi\in \GLN( G,\alpha)$ recall the map $B_\psi\in \No( G_{(G_n)},\alpha, \psi)$ given in Remark \ref{rem:desc-B_psi} and defined as $B_\psi((\xi_nG_n)_{n\geq 0}) = (\psi(\xi_{n})\psi(G_n))_{n\geq 0}$.
In the following proposition we will use the metric $d_D$ in $\Sigma^G$ given as follows: Let $D=(D_n)_{n\geq 0}$ be an increasing sequence of finite sets in $G$ such that $\bigcup_{i\geq 0} D_i=G$.
For $x,y\in \Sigma^G$, $d_D(x,y)=2^{-m}$, where $m=\max\{k\in \mathbb{N}: x|_{D_k}=y|_{D_k}\}$.

\begin{proposition}\label{prop:norm-in-Toep-1.0}
   Let $\psi\in \GLN( G,\alpha)$.
   For each $\bm{\xi}\in  G_{(G_n)}\setminus \tilde{G}$ such that $B_\psi(\bm{\xi})\in  G_{(G_n)}\setminus \tilde{G}$, it holds that the map
    \begin{align*}
     \phi_\psi:X_{f,\bm{\xi}}\to X_{{f_{\psi}},  B_\psi(\bm{\xi})}, \ \textrm{ s.t. }\  
     \phi_\psi(\sigma^g(x_{f,\bm{\xi}}))=\sigma^{\psi(g)}(x_{{f_\psi},B_\psi(\bm{\xi})})\quad \forall  g\in G,  
    \end{align*}
is a $\psi$-conjugacy.
\end{proposition}
\begin{proof}
Recall that  for every $\bm{\xi} = (\xi_n G_n)_n\in  G_{(G_n)}\setminus \tilde{G}$, the element $x_{f,\bm{\xi}}$ is given by $x_{f,\bm{\xi}}(s)=f(\tilde s\bm{\xi})$, $s\in G$.
 Define  $\tilde{\phi_\psi}:O_\sigma(x_{f,\bm{\xi}})\to X_{f_\psi,B_\psi(\bm{\xi})}$ by $\phi_\psi(\sigma^g(x_{f,\bm{\xi}}))=\sigma^{\psi(g)}(x_{f_\psi,B_{\psi}(\bm{\xi})})$, for $g\in G$. 
    We claim that $\tilde\phi_\psi$ is uniformly continuous.
    Indeed, let $n\in \mathbb{N}$ and choose $r_n\in \mathbb{N}$ in such a way $\psi^{-1}(D_{n})\subseteq D_{r_n}$.
    Let $g,h\in G$ and assume that $d_D(\sigma^g(x_{f,\bm{\xi}}), \sigma^h(x_{f,\bm{\xi}}))<\frac{1}{2^{n_r}}$.
    This implies that $f(\tilde s\tilde g\bm{\xi})=x_{f,\bm{\xi}}(sg)=x_{f,\bm{\xi}}(sh)=f(\tilde s\tilde h\bm{\xi})$ for every $s\in D_{n_r}$.
    Therefore, the definition of $f$ guarantees that for every $s\in D_{n_r}$,
    \begin{align*}
\exists k\in \mathbb{N}, \rho(\tilde s\tilde g\bm{\xi},\tilde{1_G})=\frac{1}{2^{2k}}\iff \exists l\in \mathbb{N}, \rho(\tilde s\tilde h\bm{\xi},\tilde{1_G})=\frac{1}{2^{2l}}. 
    \end{align*}
    Hence the two following statements are equivalent: for every $s\in D_{n_r}$,
    \begin{align*}
        \exists k\in\mathbb{N},&\;  s g\xi_{2k}\in G_{2k}, \textrm{ and } 
        s g\xi_{2k+1}\notin G_{2k+1},
   \end{align*}
        \textrm{if and only if }        %
    \begin{align*}
        \exists l\in\mathbb{N},\; & s  h\xi_{2l}\in G_{2l}, \textrm{ and }  s h\xi_{2l+1}\notin G_{2l+1}.
    \end{align*}
    Using that $\psi^{-1}(D_n)\subseteq D_{r_n}$ we have  for every $s\in D_n$,
    \begin{align*}
        \exists k\in\mathbb{N},\; s\psi(g)\psi(\xi_{2k})\in \psi(G_{2k}) &\textrm{ and }s\psi(g)\psi(\xi_{2k+1})\notin \psi(G_{2k+1})\\
        &\iff \\
        \exists l\in\mathbb{N},\; s\psi(h)\psi(\xi_{2l})\in \psi(G_{2l}) &\textrm{ and }
        s\psi(h)\psi(\xi_{2l+1})\notin \psi(G_{2l+1}).
    \end{align*}
    The previous condition is equivalent to say that for every $s\in D_n$, ${f_\psi}(\tilde s\widetilde{\psi(g)}B_\psi(\bm{\xi}))={f_\psi}(\tilde s\widetilde{\psi(h)}B_\psi(\bm{\xi}))$.
    Consequently, $d_D(\sigma^{\psi(g)}(x_{{f_\psi},B_\psi(\bm{\xi})}), \sigma^{\psi(h)}(x_{f_\psi,B_\psi(\bm{\xi})}))<\frac{1}{2^n}$.
    Hence, $\tilde{\phi_\psi}$ is uniformly continuous and the extension $\phi_\psi: X_{f,\bm{\xi}}\to X_{f_\psi,B_\psi(\bm{\xi})}$ is continuous.
    In an analogous way, we can define $\phi_{\psi^{-1}}:X_{f_\psi, B_\psi(\bm{\xi})}\to X_{f,\bm{\xi}}$ such that $\phi_{\psi^{-1}}(\sigma^g (x_{f_\psi,B_\psi(\bm{\xi})}))=\sigma^{\psi^{-1}(g)}x_{f,\bm{\xi}}$ for every $g\in G$. 
    Moreover, we have that $\phi_\psi\circ \phi_{\psi^{-1}}$ and $\phi_{\psi^{-1}}\circ \phi_\psi$ are the identity maps when restricted on $O_\sigma(x_{f,\bm{\xi}})$ and $O_\sigma(x_{f_\psi,B_\psi(\bm{\xi})})$, respectively. 
    Hence, $\phi_\psi:X_{f,\bm{\xi}}\to X_{f_\psi,B_\psi(\bm{\xi})}$ is a homeomorphism and by construction,  $\phi_\psi$ is a $\psi$-conjugacy. 
\end{proof}

\begin{corollary}\label{coro: belong-to-N} Let $\psi\in \GLN(G,\alpha)$ be such that $\psi(G_n) = G_n$ for each $n\in \mathbb{N}$.
Let $\bm{\xi}\in  G_{(G_n)}$ be as in Proposition \ref{prop:norm-in-Toep-1.0}.
Then, $X_{f,\bm{\xi}} = X_{f_\psi,B_\psi(\bm{\xi})}$.
Moreover, the map $\phi_\psi$ defined in Proposition \ref{prop:norm-in-Toep-1.0} belongs to $\No(X_f,\sigma,\psi)$.
\end{corollary}
\begin{proof}
 First, observe that $\rho = \rho_\psi$ since $\psi(G_n)=G_n$ for each $n\geq 0$, which, in turn, implies that $f=f_\psi$, since the definition of $f_\psi$ is the same as the one for $f$ but using $\rho_\psi$ instead of $\rho$ and in this particular case, $\rho=\rho_\psi$.
 Using Lemma \ref{lemma: equal-cocycles}, and the fact that $B_\psi(\bm{\xi})\in  G_{(G_n)}\setminus \tilde{G}$, we obtain that $X_{f,\bm{\xi}}=X_{f,B_\psi(\bm{\xi})}$.
 The second part of the corollary follows directly by Proposition \ref{prop:norm-in-Toep-1.0}, since $\phi_\psi$ is a $\psi$-conjugacy. 
\end{proof}
\begin{corollary}\label{coro: splitsequence of Xf- Aut}
    Let $G$ be a countable residually finite group and $(G_n)_{n\geq 0}$ a scale.
    Let $ (G_{(G_n)},\alpha, G)$ be the odometer associated with $(G_n)_{n\in \mathbb{N}}$ and  let $(X_f, \sigma_{X_f}, G)$ b the above subshift.
   Assume that  for each $\psi\in \GLN(G,\alpha)$ it holds that $\psi(G_n) = G_n$ for each $n\in \mathbb{N}$. Then   we have $\GLN(G,\alpha)=\GLN(G,\sigma_{X_f})$ and the short exact sequence 
    \begin{align}\label{eq: short-exact-sequence-Xf}
        1\to G\to \No(X_f,\sigma_{X_f})\to \GLN(G,\alpha)\to 1
    \end{align}
    splits.
\end{corollary}
\begin{proof}
The first part follows directly by Corollary \ref{coro: belong-to-N}.
For the second part, 
    consider the following diagram of exact short sequences, given both by \eqref{eq: shortexactsequence}, with $\AutDyn(X_f,\sigma_{X_f})\cong G$ by Corollary \ref{Coro: s=0}, and $\AutDyn(G_{(G_n)},\beta)\cong G_{(G_n)}$ by Proposition \ref{prop:right-aut},
    \begin{align*}
        \xymatrix{1\ar[r]&G\ar[d]\ar[r]&\No(X_f,\sigma_{X_f})\ar[d]^{\hat{\pi}}\ar[r]&\ar@{=}[d]\GL(G,\sigma_{X_f})\ar[r]&1\\
        1\ar[r]& G_{(G_n)}\ar[r]&\No( G_{(G_n)},\alpha)\ar[r]&\GL(G,\alpha)\ar[r]&1.}
    \end{align*}
    Note that the exact short sequence of $ G_{(G_n)}$ splits by Corollary \ref{coro:splitsequence}, where the section used $\omega: \Aut(G)\to \No( G_{(G_n)},\alpha)$ has image $\{B_\psi: \psi\in \Aut(G)\}$, with $B_\psi$ defined in Remark \ref{rem:desc-B_psi}.
    Since $\hat{\pi}$ is injective, by Proposition \ref{Prop:MEF-implies-compatible}, we only need to prove that $B_\psi\in \mbox{Im}(\hat{\pi})$ to conclude the result.
For $\psi\in \Aut(G)$, consider $\phi_\psi$ defined in Proposition \ref{prop:norm-in-Toep-1.0}. 
We claim that $\hat{\pi}(\phi_\psi)=B_\psi$. 
For every $g\in G$ we have that 
\begin{align*}
    \hat{\pi}(\phi_\psi)(\pi(\sigma^gx_{f,\bm\xi}))=\pi(\phi_\psi(\sigma^gx_{f,\bm\xi}))=\alpha^{\psi(g)}\pi(x_{f,B_\psi(\bm\xi)})=\alpha^{\psi(g)}B_\psi(\bm\xi).
\end{align*}
Now, using minimality of $G_{(G_n)}$ we obtain that $\hat{\pi}(\phi_\psi)=B_\psi$ and we conclude the proof.
   \end{proof}
 We illustrate Corollary \ref{coro: splitsequence of Xf- Aut} with various groups: the discrete Heisenberg group, ${\rm PSL}_2(\Z)$ and abelian groups.

\begin{exem}\label{example: Heisenberg}
    Let $\H_3(\Z)$ be the discrete Heisenberg group.
    That is, the group defined on $\mathbb{Z}^3$ with group multiplication given by
    \begin{align*}
        (x,y,z)\cdot(x',y',z') = (x+x',y+y',z+z'+xy').
    \end{align*}
    The center of this group is $\Cent(\H_3(\Z)) = \{(0,0,z): z\in \mathbb{Z}\}$.
    An element in $\Aut(\H_3(\Z))$ is completely characterized by the information given in  $\{(1,0,0), (0,1,0)\}$, as this set generates $\H_3(\Z)$. Recall that these isomorphism preserves the center  $\Cent(\H_3(\Z))$.
    The relation between these elements is given by 
    \begin{align}\label{eq: natural-Heisenberge-eq}
        (1,0,0)\cdot(0,1,0) = (0,1,0)\cdot(1,0,0)\cdot(0,0,1).
    \end{align}
    Let $\varphi\in \Aut(\H_3(\Z))$ be such that $\varphi(1,0,0)=(a,b,s)$, $\varphi(0,1,0) = (c,d,t)$. 
    Applying $\varphi$ to \eqref{eq: natural-Heisenberge-eq} provides 
    \begin{align*}
        (a,b,s)\cdot(c,d,t) = (c,d,t)\cdot(a,b,s)\cdot(0,0,\varepsilon), \;\varepsilon\in \{-1,1\}.
    \end{align*}
    The previous relation implies that $ad-bc = \pm 1$.
    On the other hand any element $(x,y,z)\in \H_3(\Z)$ is written as 
    \begin{align*}
        (x,y,z) = (0,1,0)^y\cdot(1,0,0)^x\cdot(0,0,1)^z,
    \end{align*}
    where $(a,b,c)^d$ denotes the product of $(a,b,c)$ with itself $d$ times in the group.
    This gives that
    \begin{align}\label{eq: Heisenberg-eq-2}
        \nonumber\varphi(x,y,z) &= (c,d,t)^y \cdot (a,b,s)^x \cdot(0,0,\varepsilon)^z\\
        &= (yc+ax, yd+bx, yt+xs+\frac{y(y-1)}{2}cd+\frac{x(x-1)}{2}ab+\varepsilon z).
    \end{align}
   Recall from \cite{LiSaUg14} that
   for each $n\in \mathbb{N}$,  the subgroup  $G_n\leqslant \H_3(\Z)$ defined by 
    \begin{align*}
        G_n = \{(s_n x, s_n y, s_n z)\in \H_3(\Z)\mid x,y,z\in \mathbb{Z}\}, \mbox{ for some fix } s_n\in \mathbb{N},
    \end{align*}
    is normal and has finite index.
    If $(s_n)_{n\geq 0}$ is a sequence of odd numbers such that $s_n\mid s_{n+1}$ (for instance, $s_n = 3^n$), then $(G_n)_{n\geq 0}$ is a scale.
    Moreover, \eqref{eq: Heisenberg-eq-2} implies that $\varphi(G_n)\subseteq G_n$ and consequently, $\varphi(G_n)=G_n$ as they have the same finite index.
    Using that $\Aut(\H_3(\Z))\cong \mathbb{Z}^2\rtimes \GL_2(\mathbb{Z})$ (see, for instance, \cite{Os15}) and Corollary \ref{coro: splitsequence of Xf- Aut}, we obtain that $\No(X_f,\sigma_{X_f})\cong \H_3(\Z)\rtimes(\mathbb{Z}^2\rtimes \GL_2(\mathbb{Z}))$.
\end{exem}

\begin{exem}\label{ex:PSL2}
    Let $\PSL_2(\mathbb{Z})= \SL_2(\Z)/\{\pm I\}$ be the group of $2\times 2$ matrices with determinant $1$ where we identify every matrix $A$ with $-A$.
    In \cite{JoTh86}, it is shown that $\Aut(\PSL_2(\mathbb{Z})) \cong \PGL_2(\Z)$, where $\PGL_2(\Z)=\GL_2(\Z)/\{\pm I\}$ and acts by conjugacy on $\PSL_2(\mathbb{Z})$. Note that $\PSL_2(\Z)$ is a subgroup of index two in $\PGL_2(\Z)$.
    Moreover, for every $n\geq 2$,
    \begin{align*}
        \Gamma(n) &= \left\{\begin{pmatrix}
            a&b\\c&d
        \end{pmatrix}\in \PSL_2(\Z)\mid a\equiv d\equiv 1\bmod{n}, b\equiv c\equiv 0\bmod{n} \right\}\\
        M(n) &= \left\{\begin{pmatrix}
            a&b\\c&d
        \end{pmatrix}\in \PSL_2(\Z)\mid a\equiv d\bmod{n}, b\equiv c\equiv 0\bmod{n} \right\}
    \end{align*}
    are normal subgroups in $\PGL_2(\Z)$.
    It is direct to check that any iso\-mor\-phism  $\psi\in \Aut(\PSL_2(\mathbb{Z}))$ preserves theses groups: $\psi(\Gamma(n))=\Gamma(n)$, and $\psi(M(n))=M(n)$.
    Therefore, for any scale $(G_n)_{n\geq 0}$ of the form $(\Gamma(s_n))_{n\in \mathbb{N}}$ or $(M(s_n))_{n\geq 0}$, the transformation group  of  the associated  odometer $(\PSL_2(\mathbb{Z})_{(G_n)},\alpha, \PSL_2(\mathbb{Z}))$ is  $\GLN(\PSL_2(\mathbb{Z}),\alpha) = \PGL_2(\Z)$.
    Moreover, it holds also true for $(X_f,\sigma, \PSL_2(\mathbb{Z}))$ by Corollary \ref{coro: splitsequence of Xf- Aut}, and hence, $\No(X_f, \sigma) \cong \PSL_2(\mathbb{Z})\rtimes \PGL_2(\Z)$.
\end{exem}

For the case $G=\mathbb{Z}^d$, we can give a precise description of the transformation groups $\GLN(\mathbb{Z}^d,\alpha)$ and $\GLN(\mathbb{Z}^d,\sigma_{X_f})$. We provide here examples where the transformation groups are all of $\Aut(\Z^d)$. We also give examples where the transformation  group is a proper subgroup.

\begin{exem}\label{ex:nonTrivialAutZd}
    Let $G = \mathbb{Z}^d$, for some $d\geq 1$. 
    Notice that any isomorphism $\psi\in \Aut(\mathbb{Z}^d) \cong \GL_d(\mathbb{Z})$ preserves subgroup of the form $s\mathbb{Z}^d$ with  $s\in \Z$: $\psi(s\Z^d)= s\Z^d$.
    Thus, if $ \Z^d_{(s_n \Z^d)}$ denotes an odometer associated with scaling sequence $s_n\mathbb{Z}^d$, for some sequence $\{s_n\}_{n\ge 0}\subset \mathbb{Z}$ (for instance, $s_n = 2^n)$, then Corollary \ref{coro: splitsequence of Xf- Aut} implies  $\No(X_f,\sigma)\cong \Z^d\rtimes \GL_d(\mathbb{Z})$.
\end{exem}

Actually, still for the case $G=\mathbb{Z}^2$,   the principal congruence subgroups 
\begin{align*}
    {\rm Cong}(m)=\left\{\begin{pmatrix}
        a&b\\c&d
    \end{pmatrix}\in \GL_2(\Z)\mid  c\equiv 0\bmod{m} \right\}, m\geq 2
\end{align*}
of $\GL_2(\Z)$ can be realized as the transformation groups of  Toeplitz subshifts.

\begin{proposition}
    For each $m\geq 2$, there exists a Toeplitz subshift $(X,\sigma, \mathbb{Z}^2)$ such that $\GLN(\Z^2,\sigma) \cong {\rm Cong}(m)$ and  $\No(X, \sigma) \cong \Z^2 \rtimes {\rm Cong}(m)$.
\end{proposition}
Before to prove it, we need to fix some notations. Recall that a finite index subgroup $\Gamma$ of $\mathbb{Z}^d$ is generated by $d$ elements $\{ v_1,\dots, v_d\}\subseteq \mathbb{Z}^d$ such that the matrix with columns $v_1,\dots, v_n$ is invertible with determinant equal to its index.
That is, $\Gamma = \langle v_1,\dots, v_d\rangle$ and  $|\det(M(v_1,\ldots, v_d))| = [G:\Gamma]$. We use the notation $M(\Gamma)$ to refer to the matrix associated with  the generators of  $\Gamma$, i.e., $M(\Gamma) = M(v_1,\dots, v_n)$. 
Note that for  any other finite index subgroup  $\Lambda= \langle w_1,\dots, w_d\rangle$,  the inclusion $\Lambda \subseteq \Gamma$ means that there exists a $d\times d$ non-singular matrix $B\in M(d,\Z)$ such that $M(\Gamma)B=M(\Lambda)$.

\begin{proof}
    Let $m\geq 2$. Let $p\in \Z^+$ be a prime number such that $(m,p)=1$. 
    Consider the scale given by the lattices  $\Lambda_n \leqslant \Z^2$, $n\ge 0$ generated by the vectors $ p^n(m , m)$ and $(p^n,0)$. The associated matrix is  $M(\Lambda_n) = p^n\begin{pmatrix}
        m&1\\m&0
    \end{pmatrix}$, $n\geq 0$.
Let $(\Z^2_{(\Lambda_n)},\alpha, \Z^2)$ be the odometer associated with the scale $(\Lambda_n)_{n\geq 0}$.

We claim that we have the following characterization: for any $2 \times 2$ matrix $A$, denoting  the  automorphism $\psi_A\in \Aut(\Z^2)$, we have
\begin{align}\label{eq: symmetry-iff-principal}
    \psi_A(\Lambda_n)=\Lambda_n, \forall n\geq 0\iff A\in {\rm Cong}(m).
\end{align}
Indeed, writing $A=\begin{pmatrix}
    a&b\\c&d
\end{pmatrix}$.   the next formula holds
\begin{align*}
    M(\Lambda_n)^{-1}AM(\Lambda_n) =& \begin{pmatrix}
            c+d&c/m\\m(a-c+b-d)&a-c
        \end{pmatrix}, \quad \forall n\geq 0.
\end{align*}
This shows that condition $A \in {\rm Cong}(m)$ is sufficient.

On the other hand, if $\psi_A\in \GLN(G,\alpha)$, then there exists $n\geq 1$ and a non-singular matrix $B\in M(2,\Z)$ such that  $AM(\Lambda_n) = M(\Lambda_0)B$.
Since $M(\Lambda_n)=p^nM(\Lambda_0)$, we have $p^n M(\Lambda_0)^{-1}AM(\Lambda_0) =B$.
Since $(m,p)=1$, we conclude that $A\in {\rm Cong}(m)$.

Therefore, we conclude the result using Corollary \ref{coro: splitsequence of Xf- Aut}.
\end{proof}

\subsection{Toeplitz subshifts with trivial automorphism group}\label{sec:ToeplTrivialAuto}
For the remainder of this section, we present as an application of Theorem \ref{Theorem: diverse-Aut} Toeplitz subshifts with automorphism group isomorphic to $\Cent(G)$ for distinct matrix groups.
As noticed below, this construction does not depend on the amenability of the groups.

First, we need the following arithmetic lemma.
\begin{lemma}\label{lemma:sequences-numbers}
    For every $l\in (0,1)$, there exist a sequence  $(t_n)_{n\geq 0}\subseteq \mathbb{Z}^+$ and a strictly increasing sequence $(m_n)_{n\geq 0}\subseteq \mathbb{N}$ such that $2^{m_{n+1}-m_n}\to\infty$ as $n\to\infty$,
     \begin{align}\label{eq:conditions-T} t_n\equiv t_{n+1}\bmod{2^{m_n}},\mbox{ and } 0\leq l-\frac{t_n}{2^{m_n}}\leq \frac{1}{2^{m_{n+1}-m_n}}+\frac{n}{2^{m_n}}.
    \end{align}
    Moreover, the second condition implies that 
    \begin{align*}
        \frac{t_n}{2^{m_n}}\to l\mbox{ and }2^{m_n}\left(l-\frac{t_n}{2^{m_n}}\right)\to\infty\mbox{ as }n\to\infty.
    \end{align*} 
    
\end{lemma}
\begin{proof}
    Let $l\in (0,1)$.
    Define the sequences $(t_n)_{n\geq 0}$ and $(m_n)_{n\geq 0}$ inductively in the following way:
    
    Let $m_0\in\mathbb{N}$   be such that $t_0:=\floor{l2^{m_0}}>0$. 
    Choose $m_1\in \N$ satisfying
    \begin{align*}
        2^{m_1-m_0}>2^{m_0}\mbox{ and }a_1 := \floor*{l2^{m_{1}}-t_0-1}>0.
    \end{align*}
    Define $t_1 := t_0+a_12^{m_0}$.
    For the general case,
    let $n\geq 1$. 
    Assume that $t_k, m_k\in \N$ were already defined for $k<n$.
    
    Pick $m_n\in \N$ so that 
       \begin{align}\label{eq: res-conditions-T}
        2^{m_{n}-m_{n-1}}>2^{m_{n-1}-m_{n-2}}\mbox{ and }a_n:=\floor*{\frac{l2^{m_{n}}-t_{n-1}-(n)}{2^{m_{n-1}}}}>0.
    \end{align}
    Define $t_{n}:=t_{n-1}+a_n2^{m_{n-1}}$.
    
    Condition \eqref{eq:conditions-T} follows directly by  \eqref{eq: res-conditions-T} and the definition of $(t_n)_{n\geq 0}$, which concludes the proof.
\end{proof}
The groups considered in the following examples are  subgroups of a general linear group $\GL_m(\mathbb{Z})$, for some $m\geq 2$.
For such subgroups, the scale $(G_n)_{n\geq 0}$ is so that  every $G_n$ is the intersection of the group we are working with and the subgroups $G_n'\leqslant \SL_m(\Z)$ given by
\begin{align}\label{eq: subgroups-gamma-n}
    G_n'=\left\{A\in \SL_m(\Z)\mid A\equiv I_m \bmod{2^n}\right\},\; n\in\mathbb{N}.
\end{align}
\textbf{Discrete Heisenberg group.}
For this example, {$\H_3(\Z)$} will denote the discrete Heisenberg group presented in Example  \ref{example: Heisenberg}.
This group can be also realised as a subgroup of $\SL_3(\mathbb{Z})$ by 
\begin{align*}
    \H_3(\Z)=\left\{M(a,b,c)\in \SL_3(\Z)\mid a,b,c\in\mathbb{Z}\right\}, \mbox{where }M(a,b,c)=\left(\begin{matrix}
        1&a&c\\
        0&1&b\\
        0&0&1
    \end{matrix}\right).
\end{align*}
Recall that the center of $\H_3(\Z)$ is given by 
\begin{align*}
    \Cent(\H_3(\Z))=\left\{M(0,0,c)\in \H_3(\Z)\mid c\in\mathbb{Z}\right\}.
\end{align*}
\begin{proposition}\label{Coro: Hei-trivial Toep}
Let $G_n = G_n'\cap \H_3(\Z)$ for every $n\geq 0$.
Then, there exists a Toeplitz subshift $(X,\sigma, \H_3(\Z))$ such that its maximal equicontinuous factor is $\H_3(\Z)_{(G_n)}$ and 
\begin{align*}
    \AutDyn(X,\sigma) \cong Z(\H_3(\Z)).
\end{align*}
\end{proposition}
\begin{proof}
    Fix any $l\in (0,1)$.
    Define the elements $\bm{\xi}=({\xi}_iG_i)_{i\geq 0},\bm{\zeta}=(\zeta_iG_i)_{i\geq 0}\in \H_3(\Z)_{(G_n)}$  by 
    \begin{align*}
     {\xi}_i = M(0,-t_n,0), \zeta_i = M(-t_n,0,0)  \mbox{ for every }m_n+1\leq i\leq m_{n+1}, n\in \N, 
    \end{align*}
    where $(t_n)_{n\geq 0}$ and $(m_n)_{n\geq 0}$ are the sequences given in Lemma \ref{lemma:sequences-numbers} for  $l\in (0,1)$.
   Assume that $A = M(a,b,c)\in  A_{\bm{\xi}}^c\cap  A_{\bm{\zeta}}^c$.
    Thus, $\bm{\eta}^{-1}\widetilde{A}\bm{\eta}=\widetilde{B}\in \widetilde{\H_3(\Z)}$, for some $B = M(o,p,q)\in \H_3(\Z)$, and $\bm{\eta}\in \{\bm{\xi},\bm{\zeta}\}$.
    For $\bm{\eta}=\bm{\xi}$, we have that $(\xi_i)^{-1}A{\xi}_i\equiv B \bmod{2^i}$ for every $i\in\mathbb{N}$.
    In particular, the coordinate $(1,3)$ of the previous relation says that \begin{align*}
        -at_n+c= q+s_{i}2^{m_n}, 
    \end{align*} 
     for some sequence  $(s_{k})_{k\geq 0}\subseteq \mathbb{Z}$,  and $i\in \N$ so that $m_n+1\leq i\leq m_{n+1}$, $n\in \N$.
    Dividing by $2^{m_n}$ and taking limits as $i$ goes to $\infty$ ($n$ also goes to $\infty$), we obtain that $-al=\lim_{i\to\infty}s_{i}$.
On the other hand,    if $a\neq 0$,  the previous  expression becomes
    \begin{align*}
        -2^{m_n}\left(\frac{s_i}{a}+\frac{t_n}{2^{m_n}}\right)=-\frac{q-c}{a}.
    \end{align*}
    Lemma \ref{lemma:sequences-numbers} implies that the left hand side of the previous equation converges to $\infty$ as $i\to\infty$, but the left hand side is constant, a contradiction.
    Therefore, $a=0$. 
An analogous treatment for $\bm{\eta}=\bm{\zeta}$ implies that $b=0$.
Therefore, we have that $ A_{\bm{\xi}}^c\cap A_{\bm{\zeta}}^c=\Cent(\H_3(\Z))$.
Applying Theorem \ref{Theorem: diverse-Aut}, for any $n_1, n_2\in \N$ and $g_1,g_2\in \H_3(\Z)$ satisfying the hypothesis, we obtain the existence of $(X_f,\sigma, \H_3(\Z))$ with $\AutDyn(X_f,\sigma) = \Cent(\H_3(\Z))\cong \Z$.
\end{proof}

\textbf{General, projective, and special linear groups.}
Let $m\geq 2$.
The general linear group $\GL_m(\mathbb{Z})$ of degree $m$  is the group containing the $m\times m$ invertible matrices over $\mathbb{Z}$  and the special linear group of degree $m$, $\SL_m(\mathbb{Z})$.
Recall that the center of these groups are given by $\{I_m,-I_m\}$, where $I_m$ denotes the identity matrix of size $m\times m$.
Now, the projective linear group of degree $m$ is given by $\PSL_m(\mathbb{Z})=\SL_m(\mathbb{Z})/\{I_m,-I_m\}$.
From before, we notice that the center of $\PSL_m(\mathbb{Z})$ is given by the class of equivalence of $I_m$, i.e., it contains only the identity element of $\PSL_m(\mathbb{Z})$.
\begin{proposition}\label{prop:TrivialCentPSL}
        Let $G$ be $\GL_m(\mathbb{Z}), \SL_m(\mathbb{Z})$, or $\PSL_m(\mathbb{Z})$.
        Let $G_n = G_n'\cap G$ for every $n\geq 0$.
Then, there exists a Toeplitz subshift $(X,\sigma, G)$ such that its maximal equicontinuous factor is $G_{(G_n)}$ and 
\begin{align*}
    \AutDyn(X,\sigma) \cong \Cent(G).
\end{align*}  
    \end{proposition}
    \begin{proof}
Let  $l\in (0,1)$, and $(t_n)_{n\geq 0}$ and $(m_n)_{n\geq 0}$ be as in Lemma \ref{lemma:sequences-numbers} for $l\in (0,1)$.
For every $1\leq i \neq j\leq m$, define $\bm{\zeta}_{i,j}=(({\zeta}_{i,j})_nG_n)_{n\geq 0}\in G_{(G_n)}$ by
\begin{align*}
    (\zeta_{i,j})_k=I_m+t_nE_{j,i}, \mbox{ for }m_n+1\leq k\leq m_{n+1}, n\geq 0,
\end{align*}
where $E_{j_0,i_0}$ is the $m\times m$ matrix with zeros everywhere except in position $(j,i)$.
In an analogous way as in the proof of Corollary \ref{Coro: Hei-trivial Toep}, we observe that $$\bigcap_{\substack{i,j=1\\i\neq j}}^m A_{\bm{\xi}_{i,j}}^c = \Cent(G).$$ 
Therefore, we conclude the proof by applying Theorem \ref{Theorem: diverse-Aut} for any ordering of the set $\{\zeta_{i,j}\mid 1\leq i,j\leq m, i\neq j\}$.
    \end{proof}


\begin{thebibliography}{99}

\bibitem{Au88}
Auslander, J. (1988).
\textit{Minimal flows and their extensions}.
North-Holland Mathematics Studies, vol.~153.
North-Holland, Amsterdam.

\bibitem{BaRoYa18}
Baake, M.; Roberts, J. A. G.; Yassawi, R. (2018).
\textit{Reversing and extended symmetries of shift spaces}.
Discrete Contin. Dyn. Syst. \textbf{38}(2), 835--866.

\bibitem{BBHLN}
Baake, M.; Bustos, A.; Huck, C.; Lemańczyk, M.; Nickel, A. (2021).
\textit{Number-theoretic positive entropy shifts with small centralizer and large normalizer}.
Ergodic Theory and Dynamical Systems, \textbf{41}(11), 3201--3226.

\bibitem{BLM}
Bustos, Á.; Luz, D.; Mañibo, N. (2023).
\textit{Admissible reversing and extended symmetries for bijective substitutions}.
Discrete \& Computational Geometry, \textbf{69}(3), 800--833.

\bibitem{BaCaRi25}
Barbieri, S.; Carrasco-Vargas, N.; Rivera-Burgos, P. (2025).
\textit{The automorphism group of a strongly irreducible subshift on a group}.
arXiv:2501.14463.

\bibitem{BoLiRu88}
Boyle, M.; Lind, D.; Rudolph, D. (1988).
\textit{The automorphism group of a shift of finite type}.
Trans. Amer. Math. Soc. \textbf{306}(1), 71--114.

\bibitem{Ca23}
Cabezas, C. (2023).
\textit{Homomorphisms between multidimensional constant-shape substitutions}.
Groups, Geometry, and Dynamics, \textbf{17}(4), 1259--1323.

\bibitem{CaPe23}
Cabezas, C.; Petite, S. (2024).
\textit{Large normalizers of $\mathbb{Z}^d$-odometers systems and realization on substitutive subshifts}.
Discrete \& Contin. Dyn. Syst. A, \textbf{44}(12), 3848-–3877

\bibitem{CeCoGo23}
Cecchi-Bernales, P.; Cortez, M. I.; Gómez, J. (2024).
\textit{Invariant measures of Toeplitz subshifts on non-amenable groups}.
 Ergodic Theory \& Dynamical Systems, \textbf{44}(11), 3186--3215.

\bibitem{Cornulier}
de Cornulier, Y. (2014).
\textit{Groupes pleins-topologiques (d’après Matui, Juschenko, Monod,...)}.
Séminaire Bourbaki, vol.~2012/2013, exp.~1061, 183.

\bibitem{CoMe16}
Cortez, M. I.; Medynets, K. (2016).
\textit{Orbit equivalence rigidity of equicontinuous systems}.
J. Lond. Math. Soc. (2) \textbf{94}(2), 545--556.

\bibitem{CoPe08}
Cortez, M. I.; Petite, S. (2008).
\textit{$G$-odometers and their almost 1-1 extensions}.
J. Lond. Math. Soc. (2) \textbf{78}(1), 1--20.

\bibitem{DoDuMaPe15}
Donoso, S.; Durand, F.; Maass, A.; Petite, S. (2016).
\textit{On automorphism groups of low complexity subshifts}.
Ergodic Theory \& Dynam. Systems \textbf{36}(1), 64--95.

\bibitem{DoDuMaPe17}
Donoso, S.; Durand, F.; Maass, A.; Petite, S. (2017).
\textit{On automorphism groups of Toeplitz subshifts}.
Discrete Anal. Paper No.~\textbf{11}, 19 pp.

\bibitem{Fu15}
Fuchs, L. (2015).
\textit{Abelian groups}.
Springer Monographs in Mathematics.
Springer, Cham.

\bibitem{FuKw19}
Fuhrmann, G.; Kwietniak, D. (2020).
\textit{On tameness of almost automorphic dynamical systems for general groups}.
Bull. Lond. Math. Soc. \textbf{52}(1), 24--42.

\bibitem{GPS95}
Giordano, T.; Putnam, I. F.; Skau, C. F. (1995).
\textit{Topological orbit equivalence and $C^*$-crossed products}.
J. Reine Angew. Math. \textbf{469}, 51--111.

\bibitem{GPS99}
Giordano, T.; Putnam, I.; Skau, C. (1999).
\textit{Full groups of Cantor minimal systems}.
Israel Journal of Mathematics \textbf{111}, 285--320.

\bibitem{GPS2019}
Giordano, T.; Putnam, I. F.; Skau, C. F. (2019).
\textit{$\mathbb{Z}^d$-odometers and cohomology}.
Groups, Geometry, and Dynamics \textbf{13}(3), 909--938.

\bibitem{HaKrSc22}
Hartman, Y.; Kra, B.; Schmieding, S. (2022).
\textit{The stabilized automorphism group of a subshift}.
Int. Math. Res. Not. IMRN \textbf{2022}(21), 17112--17186.

\bibitem{He69}
Hedlund, G. A. (1969).
\textit{Endomorphisms and automorphisms of the shift dynamical system}.
Math. Systems Theory \textbf{3}, 320--375.

\bibitem{Ho10}
Hochman, M. (2010).
\textit{On the automorphism groups of multidimensional shifts of finite type}.
Ergodic Theory \& Dynam. Systems \textbf{30}(3), 809--840.

\bibitem{HurderLukina}
Hurder, S.; Lukina, O. (2022).
\textit{Nilpotent group actions}.
Proc. Amer. Math. Soc. \textbf{150}(1), 289--304.

\bibitem{Jacobson}
Jacobson, N. (2012).
\textit{Basic algebra I}.
Courier Corporation.

\bibitem{JoTh86}
Jones, G. A.; Thornton, J. S. (1986).
\textit{Automorphisms and congruence subgroups of the extended modular group}.
J. London Math. Soc. (2) \textbf{34}(1), 26--40.

\bibitem{Gl03}
Glasner, E. (2003).
\textit{Ergodic theory via joinings}.
Mathematical Surveys and Monographs, vol.~101.
Amer. Math. Soc., Providence, RI.

\bibitem{KiRo91}
Kim, K. H.; Roush, F. W. (1991).
\textit{On the automorphism groups of subshifts}.
Pure Math. Appl. Ser. B \textbf{1}(4), 203--230.


\bibitem{Le14}
Leinster, T. (2014).
\textit{Basic category theory}.
Cambridge Studies in Advanced Mathematics, vol.~143.
Cambridge University Press, Cambridge.

\bibitem{LiSaUg14}
Lightwood, S.; Şahin, A.; Ugarcovici, I. (2014).
\textit{The structure and spectrum of Heisenberg odometers}.
Proc. Amer. Math. Soc. \textbf{142}(7), 2429--2443.

\bibitem{Li}
Li, X. (2018).
\textit{Continuous orbit equivalence rigidity}.
Ergodic Theory \& Dynamical Systems \textbf{38}(4), 1543--1563.

\bibitem{Ma91}
Margulis, G. A. (1991).
\textit{Discrete subgroups of semisimple Lie groups}.
Ergebnisse der Mathematik und ihrer Grenzgebiete (3), vol.~17.
Springer-Verlag, Berlin.

\bibitem{Ma06}
Matui, H. (2006).
\textit{Some remarks on topological full groups of Cantor minimal systems}.
Internat. J. Math. \textbf{17}(2), 231--251.

\bibitem{Med}
Medynets, K. (2011).
\textit{Reconstruction of orbits of Cantor systems from full groups}.
Bull. Lond. Math. Soc. \textbf{43}(6), 1104--1110.

\bibitem{MeTa19}
Medynets, K.; Talisse, J. P. (2019).
\textit{Toeplitz subshifts with trivial centralizers and positive entropy}.
Involve \textbf{12}(3), 395--410.

\bibitem{MerenkovSabitova}
Merenkov, S.; Sabitova, M. (2025).
\textit{On $\mathbb{Z}^d$-odometers associated to integer matrices}.
Groups, Geometry \& Dynamics \textbf{19}(1).

\bibitem{Os15}
Osipov, D. V. (2015).
\textit{The discrete Heisenberg group and its automorphism group}.
Math. Notes \textbf{98}(1-2), 185--188.

\bibitem{Renault}
Renault, J. (2008).
\textit{Cartan subalgebras in $C^*$-algebras}.
Irish Math. Soc. Bull. No.~61, 29--63.

\bibitem{RiZa10}
Ribes, L.; Zalesskii, P. A. (2010).
\textit{Profinite groups}.
Ergebnisse der Mathematik und ihrer Grenzgebiete (3), vol.~40.
Springer, Berlin.

\bibitem{Sabitova24}
Sabitova, M. (2024).
\textit{Endomorphism rings of toroidal solenoids}.
arXiv preprint arXiv:2411.11634.

\bibitem{Sabitova25}
Sabitova, M. (2025).
\textit{Computability of dimension groups}.
arXiv preprint arXiv:2509.04350.

\bibitem{SaSc25}
Salo, V. ; Schmieding, S. (2025).
\textit{Finitrary Ryan's and local $\mathcal{Q}$ entropy for $\Z^d$ subshifts}.
arXiv preprint arXiv:2503.03058.


\end{thebibliography}
\end{document}